\documentclass[10pt]{article}

\usepackage{yhmath}

\usepackage[english]{babel}

\usepackage{amsthm}
\usepackage{caption}
\usepackage{amsmath}
\usepackage{amssymb}
\usepackage{amsfonts}

\usepackage{mathdots}

\usepackage{xypic}

\newcommand{\N}{\ensuremath{\mathbb{N}}}

\newcommand{\F}{\ensuremath{\mathbb{F}}}
\newcommand{\Fq}{\ensuremath{\mathbb{F} \! _q}}

\newcommand{\Fqbar}{\ensuremath{\overline{\F} \! _q}}

\newcommand{\Fqd}{\ensuremath{\mathbb{F} \! _{q^d}}}
\newcommand{\Fqi}{\ensuremath{\mathbb{F} \! _{q^i}}}
\newcommand{\Fqis}{\ensuremath{\mathbb{F} \! _{q^i}\hspace{-0.02cm} (s)}}
\newcommand{\Fqist}{\ensuremath{\mathbb{F} \! _{q^i}\hspace{-0.02cm} (s,t)}}

\newcommand{\Fqsep}{\ensuremath{\overline{\Fq(s)}^{\operatorname{sep}}}}

\newcommand{\Fqsept}{\ensuremath{\Fqsep  \! ((t))}}

\newcommand{\kbar}{\ensuremath{\overline{k}}}
\newcommand{\kt}{\ensuremath{\overline{k}(t)}}

\newcommand{\ksep}{\ensuremath{\overline{k}^{\, \operatorname{sep}}}}

\newcommand{\G}{\ensuremath{\mathcal{G}}}
\newcommand{\Hcal}{\ensuremath{\mathcal{H}}}

\newcommand{\C}{\ensuremath{\mathbb{C}}}
\newcommand{\Z}{\ensuremath{\mathbb{Z}}}

\newcommand{\Ocal}{\ensuremath{\mathcal{O}}}
\newcommand{\ofrak}{\ensuremath{\mathfrak{o}}}
\newcommand{\Oabs}{\ensuremath{\mathcal{O}_{|\cdot|}}}
\newcommand{\pfrak}{\ensuremath{\mathfrak{p}}}

\newcommand{\mfrak}{\ensuremath{\mathfrak{m}}}
\newcommand{\qfrak}{\ensuremath{\mathfrak{q}}}
\newcommand{\Pcal}{\ensuremath{\mathcal{P}}}

\newcommand{\Gal}{\ensuremath{\mathrm{Gal}}}

\newcommand{\GL}{\operatorname{GL}}
\newcommand{\Hom}{\operatorname{Hom}}
\newcommand{\Aut}{\operatorname{Aut}}

\newcommand{\Id}{\operatorname{id}}

\newcommand{\SL}{\operatorname{SL}}
\newcommand{\Sp}{\operatorname{Sp}}

\newcommand{\SO}{\operatorname{SO}}

\newcommand{\tr}{^\mathrm{tr}}

\newcommand{\diag}{\operatorname{diag}}

\newcommand{\Mn}{\operatorname{M}_n}
\newcommand{\Quot}{\operatorname{Quot}}
\newcommand{\sigmab}{\boldsymbol{\sigma}}

\newtheorem{thm}{Theorem}[section]
\newtheorem{lemma}[thm]{Lemma}
\newtheorem{Def}[thm]{Definition}
\newtheorem{ex}[thm]{Example}

\newtheorem{prop}[thm]{Proposition}
\newtheorem{cor}[thm]{Corollary}
\newtheorem{rem}[thm]{Remark}

\usepackage{eqlist}

\title{A Difference Version of Nori's Theorem}
\author{Annette Maier 
}

\date{\today}

\begin{document}

\maketitle

{
\abstract{%
\noindent
We consider (Frobenius) difference equations over $(\Fq(s,t), \phi_q)$ where $\phi_q$ fixes $t$ and acts on $\Fq(s)$ as the Frobenius endomorphism. We prove that every semisimple, simply-connected linear algebraic group $\G$ defined over $\Fq$ can be realized as a difference Galois group over $(\Fqist,\phi_{q^i})$ for some $i \in \N$. The proof uses upper and lower bounds on the Galois group scheme of a Frobenius difference equation that are developed in this paper. The result can be seen as a difference analogue of Nori's Theorem which states that $\G(\Fq)$ occurs as (finite) Galois group over $\Fq(s)$. 
}
}

\maketitle

\section{Introduction}
In analogy to the Galois theory of polynomials (or differential equations), \textit{difference Galois theory} studies extensions generated by solutions of difference equations. A (linear) \textit{difference equation} over a \textit{difference field} $(F,\phi)$ (i.e., $F$ is a field and $\phi$ an endomorphism of $F$) is an equation of the form \[\phi(y)=Ay \] with $A \in \GL_n(F)$, $y$ a vector consisting of $n$ indeterminates and $\phi$ applied coordinate-wise to $y$. A vector $y$ with entries in a difference ring extension $R$ of $F$ (i.e., $\phi$ extends to $R$) satisfying $\phi(y)=Ay$ is called a \textit{solution} of the difference equation. The term difference equation was originally only used over the difference field $\C(z)$ with $\phi: \C(z) \rightarrow \C(z)$ given by $\phi(z)=z+1$. A classic example is the one-dimensional difference equation $\phi(y)=zy$ over $\C(z)$ which is solved by the Gamma function. \\
\\
An equivalent concept to difference equations is that of \textit{difference modules}. An $n$-dimensional difference module $(M,\Phi)$ over a difference field $(F,\phi)$ is an $n$-dimensional $F$-vector space together with a difference structure $\Phi: M \rightarrow M$ that is given by a matrix $D \in \GL_n(F)$ (with respect to a fixed basis $\mathcal{B}$ of $M$): After identifying $M$ with $F^n$ via the basis $\mathcal{B}$, $\Phi(x)=D\phi(x)$ for $x \in M$; in other words, $\Phi$ acts \textit{semilinearly} on $M$ and $D$ collects the images of $\mathcal{B}$ in its columns. A \textit{solution} of $(M,\Phi)$ is an element $x \in M\otimes_F R$ for some difference ring extension $(R,\phi)$ of $(F,\phi)$ such that $\Phi(x)=x$, where $\Phi$ acts on $M\otimes_F R$ via $\Phi \otimes \phi$. Note that the solutions of $(M,\Phi)$ are in bijection to the solutions of the difference equation $\phi(y)=D^{-1}y$. There is the notion of a \textit{Picard-Vessiot ring} which is in some sense a ``smallest'' ring extension $R$ of $F$ together with an extension of $\phi$ such that there exists a full set of solutions in $M\otimes_F R$. In case the \textit{constants} $C$ of $F$ (the elements fixed by $\phi$) are algebraically closed, there always exists a Picard-Vessiot ring. The \textit{difference Galois group} can then be defined as the group of automorphisms of $R$ that leave $F$ (pointwise) invariant and commute with $\phi$; it turns out to be a linear algebraic group defined over $C$ (as developed in \cite{PutSinger}). This can be generalized to the case of an arbitrary field of constants $C$, leading to difference Galois groups that are affine group schemes defined over $C$ provided that there exists a Picard-Vessiot ring. The difference Galois group of an $n$-dimensional difference module can be embedded into $\GL_n$.\\
\\
Similar to the inverse problem in classical Galois theory, it is a natural question to ask which affine group schemes defined over $C$ occur as Galois groups of some difference module over the fixed base field $F$ with fixed endomorphism $\phi$. For example, if $F=\C(z)$ with $\phi$ given by $\phi(z)=z+1$ as above, it has been conjectured that a linear algebraic group $\G$ over $\C$ is a difference Galois group if and only if the quotient $\G / \G^{0}$ by the identity component is cyclic  (it is known that this is true for $\G=\G^{0}$ and that the condition on $\G / \G^{0}$ is necessary - see \cite{PutSinger}).\\
\\
Let now $F=\Fq(s,t)$ be a function field in two variables over the finite field $\Fq$ with $\phi=\phi_q$ acting trivially on $\Fq(t)$ and mapping $s$ to $s^q$. Then the constants of $F$ are $C=\Fq(t)$. Difference modules over $(F,\phi_q)$ are also called Frobenius modules. The \textbf{main result} of this paper is that every semisimple and simply-connected group $\G$ that is defined over $\Fq$ occurs as a $\phi_{q^i}$-difference Galois group over $\Fqi(s,t)$ for some $i \in \N$, (Theorem \ref{result}). The number $i$ has to be chosen in such a way that the following holds:
\begin{itemize}
\item $\G$ splits over $\Fqi$ and there exists a regular element $g_0 \in \G(\Fqi)$ contained in a maximal torus that splits over $\Fqi$
\item a certain place $\pfrak$ of $\Fq(s)$ (depending on $g_0$) splits into places of degree $1$ inside $\Fqis$. 
\end{itemize}
It should be mentioned that in case $F=\Fqbar(s)((t))$ with $\phi_q$ acting coefficient-wise as the Frobenius endomorphism on $\Fqbar(s)$ (hence the constants of $F$ equal $\Fq((t))$), the inverse problem has been solved by Matzat. Namely, Theorem 2.3 in \cite{Matzat2} implies that any linear algebraic group defined over $\Fq((t))$ occurs as a difference Galois group over $\Fqbar(s)((t))$. However, this result is based on taking $t$-adic limits, so the proof cannot be transferred to our non-complete base field $\Fq(s)(t)$ or even $\Fqbar(s)(t)$.\\
\\
Our approach instead uses upper and lower bound techniques as follows. First, we give a necessary condition for the existence of a Picard-Vessiot ring of a Frobenius module (see Theorem \ref{thmexistenceofsol} together with Theorem \ref{Pap4.2.5}). Let $(M,\Phi)$ be an $n$-dimensional Frobenius module with representing matrix $D \in \GL_n(\Fq(s,t))$ satisfying this condition and let $\Hcal\leq \GL_n$ denote its Galois group. We show that given a connected linear algebraic group $\G\leq \GL_n$ defined over $\Fq$, we have $\Hcal \leq \G$ if $D$ is contained in $\G(\Fq(s,t))$ (Theorem \ref{ubthm} together with Proposition \ref{closedemb}). This gives an upper bound on the difference Galois group $\Hcal$. We also develop a lower bound criterion that yields explicit elements (depending on $D$) that are contained in $\Hcal$ up to conjugacy (see Theorem \ref{uschr}). These criteria can be used to construct Frobenius modules with given Galois group $\G$: Find a matrix $D \in \G(\Fq(s,t))$ (the representing matrix of the Frobenius module we are looking for) that meets the assumptions for the existence of a Picard-Vessiot ring such that any conjugates of the elements provided by our lower bound criterion generate $\G$ (or a dense subgroup of $\G$); the latter condition requires a certain knowledge of how to generate $\G$. In this way, we derived explicit Frobenius modules with difference Galois groups $\SL_n$, $\Sp_{2d}$, $\SO_n$ and the Dickson groups $G_2$ in (\cite{Diss}). \\
\\
Let now $\G$ be an arbitrary semisimple and simply-connected group defined over $\Fq$. Nori's theorem (\cite{Nori}) asserts that $\G(\Fq)$ can be realized as finite Galois group over $\Fq(s)$. Using our three criteria explained above, we can ``lift'' this finite extension to a Frobenius module over $\Fq(s,t)$ with Galois group $\Hcal \leq \G$ defined over $\Fq(t)$ such that every element in $\G(\Fq)$ occurs as constant coefficient matrix of some element inside $\Hcal(\Fq[[t]])$, and such that $\Hcal$ contains a certain conjugate of a maximal torus $T$ of $\G$ that is defined over $\Fq$. By passing from $\Fq$ to $\Fqi$, we may assume that $T$ splits over $\Fq$. Then we can use the structure theory of split reductive linear algebraic groups to show that any closed subgroup $\Hcal \leq \G$ as above equals $\G$ (Theorem \ref{arbgen}).\\
\\ 
We can also lift our difference modules from $\Fq(s,t)$ to $\overline{\Fq(s)}(t)$ to get difference modules with the same Galois group. As a result we obtain \textit{rigid analytically trivial pre-$t$-motives} with arbitrary semisimple, simply-connected Galois groups. The category of rigid analytically trivial pre-$t$-motives contains the category of $t$-motives, which is of importance in the number theory of function fields (see for example \cite{Papanikolas}).\\
\\
The paper is organized as follows. Section \ref{sec2} provides some background on the Galois theory of difference modules (with not necessarily algebraically closed fields of constants) collecting all statements used later. In Section \ref{sec3}, we set up some notation that will be used throughout all following sections. In Section \ref{sec4}, we develop techniques to guarantee that a Frobenius module has a certain difference Galois group. Specifically, Section \ref{EX} is concerned with the existence of Picard-Vessiot rings, while Sections \ref{UB} and \ref{LB} provide upper and lower bounds for difference Galois groups. Section \ref{GE} deals with finding generators of reductive groups that respect a certain conjugacy. Using the results from Sections \ref{sec4} and \ref{GE}, we can prove the main result in Section \ref{GEN}. In the last section, we translate the result to the language of pre-$t$-motives. \vspace{-0.3cm}
\subsection*{Acknowledgements} The author would like to thank J. Hartmann for many valuable suggestions and discussions, as well as D. Harbater, B.\,H. Matzat and M. Wibmer for helpful comments.

\section{Basics of Difference Galois Theory}\label{sec2}
In this section we give a short introduction to the Galois theory of difference modules (in other words, the Galois theory of (linear) difference equations). The standard reference is \cite{PutSinger}; unfortunately, the authors restrict themselves to algebraically closed fields of constants and surjective difference homomorphisms (``inversive'' difference fields). Arbitrary fields of constants (but still in the inversive case) are treated in \cite{AmanoMas}. A more general approach is taken in \cite{Wibmer} which allows certain non-linear difference equations. Most of the statements quoted in this section can be proven similarly to the classical case (inversive and algebraically closed fields of constants). They also follow from the more general theory in \cite{Wibmer}. Direct proofs can be found in \cite{Diss}.
\begin{Def}
A \emph{difference ring} $(R,\phi)$ is a commutative ring $R$ equipped with a ring homomorphism $\phi \colon R \rightarrow R$. A \emph{difference field} is a difference ring which is a field. The \emph{constants} $C_R$ of a difference ring $R$ are the elements of $R$ fixed by $\phi$. A \emph{difference ideal} of a difference ring $R$ is a $\phi$-stable ideal of $R$ and $R$ is called a \emph{simple difference ring} if its only difference ideals are $(0)$ and $R$. In this case, $C_R$ is a field. If $R$ and $S$ are difference rings, a ring homomorphism $\sigma \colon R \rightarrow S$ commuting with the difference structure on $R$ and $S$ is called a \emph{difference homomorphism}. The set of all such is denoted by $\Hom^{\phi}(R,S)$.
\end{Def}

\begin{ex} 
Let $q$ be a prime power and consider $\Fq(s,t)$. Let $\phi_q$ be the homomorphism on $\Fq(s,t)$ fixing $t$ and restricting to the ordinary Frobenius endomorphism on $\Fq(s)$, i.e., $\phi_q(s)=s^q$. Then $\Fq(s,t)$ is a difference field extension of $\Fq(s)$, with constants $\Fq(t)$. Note that $\phi_q$ is not an automorphism.
\end{ex}

\begin{Def}\label{defdiffmod}
Let $(F,\phi)$ be a difference field. A \emph{difference module} (or $\phi$-module, for short) over $F$ is a finite dimensional $F$-vector space $M$ together with a $\phi$-semilinear map $\Phi \colon M \rightarrow M$, (i.e., $\Phi$ is additive and for any $\lambda \in F$ and $x \in M$ we have $\Phi(\lambda x)=\phi(\lambda)\Phi(x)$) such that there exists a representing matrix $D$ contained in $\GL_n(F)$, where $n=\dim_F(M)$. A \emph{representing matrix} $D$ is defined as follows: With respect to a fixed basis of $M$, the action of $\Phi$ is completely described by the images of the basis elements. The representing matrix $D$ (with respect to this basis) collects these images in its columns. Conversely, every $D \in \GL_n(F)$ gives rise to an $n$-dimensional difference module.\\
 A \emph{fundamental (solution) matrix} for $M$ in some difference ring extension $R \geq F$ is defined to be an element $Y \in \GL_n(R)$ such that $D\phi(Y)=Y$ holds, where $\phi$ is applied coordinate-wise. There exists a $\Phi$-invariant basis of $M\otimes_F R$ if and only if there exists a fundamental solution matrix $Y \in \GL_n(R)$. Indeed, the elements in $M$ represented by the columns of a fundamental matrix are $\Phi$-invariant. 
\end{Def}

We now present the notion of Picard-Vessiot rings of difference equations (which do not necessarily exist if the field of constants is not algebraically closed).

\begin{Def}
Let $(F, \phi)$ be a difference field with constants $C$ and let $(M,\Phi)$ be a difference module over $(F,\phi)$ with representing matrix $D\in\GL_n(F)$. An extension of difference rings $R/F$ is called a Picard-Vessiot ring for $M$ if the following holds:
\begin{itemize}
\item $R$ is a simple difference ring.
\item The field of constants of $R$ is $C$.
\item There exists a fundamental matrix $Y \in \GL_n(R)$, i.e., $D\phi(Y)=Y$.
\item $R$ is generated as $F$-algebra by $\{Y_{ij}, \det(Y)^{-1} \ | \ 1 \leq i,j \leq n \}$.
\end{itemize}
We will use the notation $F[Y,Y^{-1}]:=F[Y_{ij}, \det(Y)^{-1} \ | \ 1 \leq i,j \leq n]$. 
\end{Def}

 The next theorem guarantees the existence of Picard-Vessiot rings provided there exists a fundamental matrix contained in a difference field extension with no new constants. 

\begin{thm} \label{Pap4.2.5}
Let $(F,\phi)$ be a difference field with field of constants $C$ and let $M$ be a difference module over $F$. Assume that $L/F$ is a difference field extension such that
\begin{enumerate}
 \item The field of constants of $L$ is $C$, 
\item There exists a fundamental matrix $Y \in \GL_n(L)$, i.e., $D\phi(Y)=Y$,
\end{enumerate}
Then $R:=F[Y,Y^{-1}] \subseteq L$ is a Picard-Vessiot ring for $M$ and $R$ is the only Picard-Vessiot ring for $M$ that is contained in $L$.
\end{thm}

\begin{Def}
In the situation as in Theorem \ref{Pap4.2.5}, i.e., if $R$ is an integral domain, we call $\Quot(R)$ a \emph{Picard-Vessiot extension} for $M$. 
\end{Def}

\begin{rem}
Other than in differential theory, the existence of a Picard-Vessiot ring does not imply the existence of a Picard-Vessiot extension, since a difference Picard-Vessiot ring is not necessarily a domain (we only know that it is reduced). It is therefore more natural to work with the total quotient rings of Picard-Vessiot rings instead of the field of fractions. However, the (explicitly constructed) Picard-Vessiot rings of the difference modules considered in this paper will be domains, so Theorem \ref{Pap4.2.5} will be sufficient for our purpose. 
\end{rem}

We now give a construction of the Galois group scheme $\G$ of a Picard-Vessiot ring $R$, which turns out to be a linear algebraic group under certain separability assumptions. We will not assume our Picard-Vessiot ring to be integral.

\begin{Def} Let $(F,\phi)$ be a difference field with field of constants $C$ and let $R$ be a Picard-Vessiot ring for some difference module over $F$. We write $\underline{Aut}(R/F)$ for the functor from the category of $C$-algebras to the category of groups sending a $C$-algebra $S$ to the group $\Aut^{\phi}(R\otimes_C S/ F\otimes_C S)$ of difference automorphisms fixing $F \otimes_C S$. Note that we consider $R \otimes_C S$ as difference ring via $\phi \otimes \Id$. 
\end{Def}

The key ingredient to show that $\underline{Aut}(R/F)$ is representable is the following proposition.
\begin{prop}\label{torsorprep}
Let $(F,\phi)$ be a difference field with constants $C$ and let $R/F$ be a Picard-Vessiot ring for a difference module over $F$. Then we have an $R$-linear isomorphism of difference rings
$$R\otimes_F R \cong R \otimes_C C_{R\otimes_F R}, $$ where $R\otimes_F R$ and $R \otimes_C C_{R\otimes_F R}$ are considered as difference rings via $\phi\otimes_F \phi$ and $\phi\otimes_C \Id$, resp. 
\end{prop}

\begin{thm}\label{representability}
The group functor $\underline{Aut}(R/F)$ is represented by the $C$-algebra $C_{R \otimes_F R}$, and is thus an affine group scheme over $C$. If moreover $R$ is separable over $F$, then $\G=\underline{Aut}(R/F)$ is a linear algebraic group over $C$, that is, an affine group scheme of finite type over $C$, such that $\G \times_C \overline{C}$ is reduced (i.e., $\G$ is ``geometrically reduced"). 
\end{thm}

\begin{Def} \label{defgalois} 
Let $(F,\phi)$ be a difference field with field of constants $C$ and let $(M,\Phi)$ be a difference module over $(F,\phi)$ with a Picard-Vessiot ring $R$. Then we call $\underline{Aut}(R/F)$ the Galois group scheme of $M$ (with respect to $R$, which is not unique, in general). Two different Picard-Vessiot rings for the same difference module yield Galois groups that are isomorphic over an algebraic closure of $C$.
\end{Def}

As a corollary to Proposition \ref{torsorprep}, we get the well-known identity between transcendence degree of Picard-Vessiot extensions and dimension of their Galois group scheme:
\begin{cor}\label{torsorcor}
Let $(F,\phi)$ be a difference field with field of constants $C$ and let $R$ be a Picard-Vessiot ring for a difference module over $(F,\phi)$ with Galois group scheme $\G$. Then $R\otimes_F \overline{F} \cong C[\G] \otimes_C \overline{F}$, where $\overline{F}$ denotes an algebraic closure of $F$. In particular, $\operatorname{dim}(R)=\dim(\G)$, where $\operatorname{dim}(R)$ denotes the Krull dimension of $R$. 
\end{cor}

An explicit linearization of $\G=\underline{Aut}(R/F)$ can be given using a fundamental solution matrix:

\begin{prop}\label{closedemb}
 Let $R$ be a Picard-Vessiot ring for a difference module over a difference field $(F,\phi)$. Let $C$ be the field of constants and let $\G$ be the Galois group scheme. Assume further that $R$ is separable over $F$. Then there is a closed embedding $\rho \colon \G \hookrightarrow \GL_n$ of linear algebraic groups such that for any $C$-algebra $S$, we have
\[\rho_S \colon \G(S)=\Aut^\phi(R\otimes_C S/F\otimes_C S)\rightarrow \GL_n(S), \ \sigma \mapsto Y^{-1}\sigma(Y). \]
\end{prop}

Proposition \ref{closedemb} becomes particularly useful for obtaining upper bounds on the Galois group $\G$: Let $R/F$ be a separable Picard-Vessiot ring with Galois group scheme $\G$. Assume that there exists a fundamental solution matrix $Y$ that is contained in $\tilde \G(R)$ for some closed subgroup $\tilde \G \leq \GL_n$ defined over $C$. Then for all $\gamma \in \Aut(R\otimes_C S/F\otimes_C S)$, $\gamma(Y)$ is contained in $\tilde \G(R\otimes_C S)$ and $\G_{R/F} \cong \rho(\G_{R/F} )$ is thus contained in $\tilde \G$. \\
\\
The following proposition consists of one direction of a Galois correspondence for difference modules and can be proven directly using Proposition \ref{torsorprep}. 

\begin{prop}\label{Invarianten}
Let $(R,\phi)$ be a Picard-Vessiot ring over a difference field $(F,\phi)$ with Galois group scheme $\G$. Let $\frac{a}{b}$ be an element in the total quotient ring of $R$ (the localization at the set of all non zero divisors of $R$). If $\frac{a}{b}$ is functorially invariant under the action of $\G$, i.e., for every $C$-algebra $S$ and every $\sigma \in \Aut^\phi(R\otimes_C S/ F\otimes_C S)$ we have 
\[\sigma(a\otimes_C 1)\cdot(b\otimes_C 1)=(a \otimes_C 1)\cdot \sigma(b\otimes_C 1), \] then $\frac{a}{b}$ is contained in $F$. 
\end{prop}

As a consequence, we get the following lemma. 
\begin{lemma}\label{eindim_UM}
Let $(M, \Phi)$ be an $m$-dimensional difference module over a difference field $(F,\phi)$, with Picard-Vessiot extension $E$, fundamental matrix $Y \in \GL_m(E)$ and Galois group scheme $\Hcal \leq \GL_m$. Suppose that there exists a $0 \neq w \in C_F^m$ that spans an $\Hcal$-stable line, i.e., for any $C_F$-algebra $S$, we have $\Hcal(S)\cdot w \subseteq S\cdot w$. Then there exists an $\alpha \in E^{\times}$ such that $v:=\alpha Y \! \!  \cdot \!  w \in F^m \cong M$ and $N:=F \! \cdot \! v$ defines a $\Phi$-stable submodule of $M$. 
\end{lemma}

We conclude this section with a theorem concerning base change of Picard-Vessiot rings. It can easily be proven using Corollary \ref{torsorcor}. 

\begin{thm}\label{basechangethm}
Let $(F,\phi)$ be a difference field and let $R/F$ be a Picard-Vessiot ring for a difference module $M$ over $F$ such that its Galois group scheme $\G$ is a connected linear algebraic group. \\
If $(\tilde F, \phi)$ is an algebraic difference field extension of $F$ such that $\tilde F$ and $R$ are both contained in some common difference field $L$ without new constants, then $R\otimes_F\tilde F$ is a Picard-Vessiot ring over $(\tilde F,\phi)$ for $M\otimes_F \tilde F$ with Galois group scheme $\G$. 
\end{thm}

\section{Notation}\label{notation} \label{sec3}
We define difference fields $k(t)\subseteq K(t)\subseteq L$ with field of constants $\Fq(t)$. 
\begin{description}
\item[$q$] a power of a prime $p$.
\item[$k$] a field containing $\Fq$ with a fixed (non-trivial) non-archimedian absolute value $|\cdot|$. 
\item[$K$] the completion of an algebraic closure of the completion of $k$ with respect to $|\cdot|$. Note that $K$ is algebraically closed.
\item[$\overline{k}$] the algebraic closure of $k$ contained in $K$.
\item[$\ksep$] the separable algebraic closure of $k$ contained in $K$.
\item[$\Ocal_{|\cdot|}$] the valuation ring in $K$ corresponding to $|\cdot|$.
\item[$\mfrak$] the maximal ideal inside $\Ocal_{|\cdot|}$.
\item[$K\{t\}$] the ring of power series that converge on the closed unit disk: $K\{t\}:=\{\sum_{i=0}^{\infty}\alpha_it^i \in K[[t]]\ | \ \lim \limits_{i\rightarrow \infty}|\alpha_i|=0\}$.
\item[$L$] the field of fractions of $K\{ t \}$: $L=\operatorname{Quot}(K\{t\})$.
\item[$(K,\phi_q)$] $\phi_q \colon K \rightarrow K, \ \lambda \mapsto \lambda^q$ is the ordinary Frobenius endomorphism on $K$. The field of constants then equals $C_K=\Fq$. 
\item[$(K((t)),\phi_q)$] The action of $\phi_q$ on Laurent series over $K$ is defined coefficient-wise, i.e. $C_{K((t))}=\Fq((t))$.
\item[$(L,\phi_q)$] The action of $\phi_q$ on $K((t))$ induces a homomorphism on $L \subseteq K((t))$. The field of constants then equals $C_L=\Fq(t)$.
\item[$(k(t),\phi_q)$] the difference structure on $k(t)$ is induced by that on $K(t) \subseteq L$, i.e., $\phi_q$ only acts on the coefficients of a rational function. Then $C_{k(t)}=\Fq(t)$ holds.
\item[$M_n$] $n\times n$-matrices.
 \end{description}

\begin{ex}
The standard examples are $k=\Fq(s)$, a function field in one variable with an $s$-adic absolute value $|\cdot|$ or $k=\overline{\Fq(s)}$.\\ One might also consider function fields $\Fq(s_1,\dots,s_n)$ in several variables with $|\cdot|$ for instance an $s_1$-adic absolute value. 
\end{ex}

\begin{rem}\label{remsep}
\begin{enumerate}
 \item Note that $L/k(t)$ is usually not a separable extension (as $K/k$ might not be separable), and thus $(\Fq(t), k(t), L)$ is not a $\phi_q$-admissible triple as defined in \cite[4.1.]{Papanikolas}. 
\item  Sometimes people work with the inverse $\sigma$ of $\phi_q$ instead of $\phi_q$, but since this is not defined on our base field $k(t)$ if $k$ is not perfect, we prefer to work with $\phi_q$, instead. 
\end{enumerate}
\end{rem}

\section{Bounds on Difference Galois Groups} \label{sec4}
In this section, we prove upper and lower bounds on the Galois group of a difference module over $(k(t), \phi_q)$ that only depend on a representing matrix of the difference module. These criteria are aimed to construct difference modules with prescribed Galois group schemes. 
\subsection{Existence of Picard-Vessiot Extensions} \label{EX}
We first give a criterion (Theorem \ref{thmexistenceofsol}) that provides us with a fundamental solution matrix $Y \in \GL_n(L)$ to a given difference module over $(k(t), \phi_q)$. Note that $Y \in \GL_n(L)$ ensures the existence of a Picard Vessiot ring (see Theorem \ref{Pap4.2.5}). 
We start with a multidimensional version of Hensel's Lemma. For $m \in \N$, let $|| \cdot ||$ denote the maximum norm on $K^m$ induced by $|\cdot|$: \[||(a_1,\dots,a_m)||:=\max\{|a_i| \ | \ 1 \leq i \leq m \}.\]
\begin{lemma}[Hensel's Lemma]\label{Hensel}  
 Let $f_1,\dots,f_m \in \Ocal_{|\cdot|}[X_1,\dots,X_m]$ be a system of $m$ polynomials in $m$ variables with coefficients in $\Oabs$. Assume that there exists a vector $b=(b_1,\dots,b_m) \in \Ocal_{|\cdot|}^m$ such that $||(f_1(b),\dots,f_m(b))||<|\det(J_b)|^2$, where $J_b=(\frac{\partial f_i}{\partial X_j}(b))_{i,j}$ denotes the Jacobian matrix at $b$. Then there is a unique $a \in \Ocal_{|\cdot|}^m$ satisfying $f_i(a)=0$ for all $1 \leq i \leq m$ and \[||a-b|| =\frac{|| J_b^{*}\cdot(f_1(b),\dots,f_m(b))^{\tr}||}{|\det(J_b)|},\] where $J_b^{*}$ denotes the adjoint matrix of $J_b$.
\end{lemma}
\noindent This version of Hensel's Lemma is sometimes also called multi-dimensional Newton's Lemma. It holds for all henselian fields (note that $K$ is henselian as it is complete with respect to a rank one valuation). For a proof, see Theorem 23 and 24 of \cite{Kuhlmann}.

\begin{cor}\label{Henselqadd}
Let $A$ and $B$ be contained in $\Mn(\Ocal_{|\cdot|})$ and consider the system of polynomial equations 
\[AY^q-Y+B=0, \] where $Y=(Y_{ij})_{i,j \leq n}$ consists of $n^2$ indeterminates and $Y^q:=(Y_{ij}^q)_{i,j}$. Assume that there exists a $Y' \in \Mn(\Ocal_{|\cdot|})$ such that $||A(Y')^q-Y'+B||<1$. Then there exists a unique solution $Y\in \Mn(\Ocal_{|\cdot|})$ of $AY^q-Y+B=0$ such that $||Y-Y'||= ||A(Y')^q-Y'+B||$.
\end{cor}
\begin{proof}
This is an immediate consequence of Lemma \ref{Hensel}. Indeed, let 
$f_{rs} \in \Ocal_{|\cdot|}[Y_{ij}\ | \ 1\leq i,j \leq n]$, $1\leq r, s \leq n$ be the system of polynomials defining $AY^q-Y+B=0$ and let $A_{rs}, B_{rs}$ be the coordinates of $A$ and $B$ ($1 \leq r,s \leq n$). Then \[f_{rs}=\sum_{m=1}^{n}A_{rm}Y_{ms}^q-Y_{rs}+B_{rs},\] hence $\frac{\partial f_{rs}}{\partial Y_{ij}}=-\delta_{(i,j),(r,s)}$. This means that $J_b$ equals the negative of the $n^2\times n^2$ identity matrix for all $b \in M_n(K)$, so $Y'$ meets the assumptions of the element $b$ in Lemma \ref{Hensel}. Also, up to a sign, $J_{Y'}^{*}$ equals the identity matrix, so the claim follows. 
\end{proof}

\begin{thm} \label{thmexistenceofsol}
Let $D=\sum_{l=0}^{\infty} D_lt^l \in \GL_n(\Ocal_{|\cdot|}[[t]])$ (with $D_l \in \Mn(\Ocal_{|\cdot|})$) be such that there exists a $\delta < 1$ with 
\[ ||D_l|| \leq \delta^l \] for all $l \in \N$. Then there exists a fundamental matrix $Y \in \GL_n(L)$ for $D$, i.e., $D\phi_q(Y)=Y$. More precisely, $Y=\sum_{l=0}^{\infty}Y_lt^l \in \GL_n(\Oabs[[t]])$ with $Y_l \in \Mn(\Ocal_{|\cdot|})$ satisfying $||Y_l|| \leq \delta^l$ for all $l \in \N$. 
\end{thm}
\begin{proof}
Observe that $D\phi_q(Y)=Y$ is equivalent to 
\[D_0Y_l^q+D_1Y_{l-1}^q+\dots+D_lY_0^q=Y_l \ \ \mbox{ for all } l \in \N. \]
We define $(Y_l)_{l\geq 0}$ inductively. For $l=0$, we need to solve $D_0\phi_q(Y_0)=Y_0$. The Lang isogeny (see \cite[V.16.4]{Borel}) asserts that such a $Y_0$ exists inside $\GL_n(K)$, as $K$ is algebraically closed. Then $Y_0^q=D_0^{-1}Y_0$ holds, hence $\Ocal_{|\cdot|}[(Y_0)_{ij} \ | \ 1 \leq i,j \leq n ]$ is finitely generated as an $\Ocal_{|\cdot|}$-module, since $D_0 \in \GL_n(\Ocal_{|\cdot|})$. Therefore, all entries of $Y_0$ are integral over $\Ocal_{|\cdot|}$ and as $\Ocal_{|\cdot|}$ is integrally closed inside $K$, we conclude that $||Y_0||\leq 1=\delta^0$ holds. On the other hand, $D_0\phi_q(Y_0)=Y_0$ implies $\det(D_0)\det(Y_0)^q=\det(Y_0)$, hence $\det(Y_0)^{-1}$ is integral over $\Oabs$ which implies $\det(Y_0) \in \Oabs^{\times}$ and therefore $Y_0 \in \GL_n(\Oabs)$.\\
\\
Now suppose that $Y_0,\dots,Y_{l-1}$ have been chosen such that for all $1 \leq i \leq l-1$, $||Y_i|| \leq \delta^i$ and 
$D_0Y_i^q+D_1Y_{i-1}^q+\dots+D_iY_0^q=Y_i$ holds. We claim that we can find $Y_l \in \Mn(\Ocal_{|\cdot|})$ such that $D_0Y_l^q+D_1Y_{l-1}^q+\dots+D_lY_0^q=Y_l$ and $||Y_l|| \leq \delta^l$. Set $A:=D_0 \in \GL_n(\Ocal_{|\cdot|})$ and $B:=D_1Y_{l-1}^q+\dots+D_lY_0^q \in \Mn(\Oabs).$ We have to find a solution to the polynomial system of equations 
\[AY^q-Y+B=0.\] 
We have 
\begin{eqnarray*} 
||B||&=&||D_1Y_{l-1}^q+\dots+D_lY_0^q|| \\
&\leq& \max\{||D_iY_{l-i}^q|| \ | \ 1 \leq i \leq l \} \\
&\leq& \max \{||D_i||\cdot ||Y_{l-i}||^q \ | \ 1 \leq i \leq l \} \\
&\leq& \max \{\delta^i\cdot \delta^{(l-i)q} \ | \ 1 \leq i \leq l \} \\
&\leq& \delta^l, 
\end{eqnarray*} where we used that the maximum norm $||\cdot||$ coming from a non-archimedian absolute value is sub-multiplicative with respect to the matrix multiplication. 
Let $\theta \in \Ocal_{|\cdot|}$ be an element such that $|\theta| \leq \delta$ and set 
$Y_l'=\theta^l \cdot I_n$, where $I_n$ denotes the identity matrix. Then we have 
\[ ||A(Y_l')^q-Y_l'+B||\leq \max\{||A||\cdot||Y_l'||^q, ||Y_l'||, ||B||\}\leq \delta^l<1.\] Hence by Corollary \ref{Henselqadd}, there exists an element $Y_l \in \Mn(\Ocal_{|\cdot|})$ such that $AY_l^q-Y+B=0$ and $||Y_l-Y_l'||=||A(Y_l')^q-Y_l'+B||\leq \delta^l$. As $||Y_l'|| \leq \delta^l$, we conclude $||Y_l|| \leq \delta^l$. \\
\\
The resulting matrix $Y=\sum_{l=0}^{\infty}Y_l t^l \in \Mn(K\{t\})\subseteq \Mn(L)$ satisfies \\ $D\phi_q(Y)=Y$ and $||Y_l|| \leq \delta^l$ for all $l \in \N$. In particular, $Y \in \Mn(\Oabs[[t]])$ and we have seen above that $Y_0 \in \GL_n(\Oabs)$, hence $Y \in \GL_n(\Oabs[[t]])$.
\end{proof}

\subsection{An Upper Bound Theorem}\label{UB}

Let $F$ be a difference field with field of constants $C_F$ and let $\G$ be a connected linear algebraic group defined over $C_F$. For algebraically closed fields of constants it is well known that the Galois group of a difference module is contained in $\G$ if its representing matrix is contained in $\G(F)$ (see for example \cite[Prop. 1.31]{vdPutSinger}). In our setup of difference fields with a valuation and fields of constants $\Fq(t)$, we prove such a criterion under certain assumptions (see Theorem \ref{ubthm} below). The strategy is to show that there exists a fundamental matrix contained in $\G$ if there exists one in $\GL_n$. This implies that the Galois group scheme is contained in $\G$ (see Prop. \ref{closedemb}).\\
\\
The Chevalley Theorem \ref{chevalley} has played an important role in solving the inverse problem in differential Galois theory (with algebraically closed constants). It has been used in characteristic zero (see \cite{MitschiSinger}) as well as in the iterative differential case (see \cite{Matzatskript}). We use it in the proof of both the upper and the lower bound theorem.
\begin{thm}\label{chevalley}(Chevalley, see \cite[Theorem 5.5.3]{Springer}) \\
Let $\G$ be a linear algebraic group over an algebraically closed field $F$ and $\mathcal{H}$ a closed subgroup, both defined over a subfield $F_1\subseteq F$. Then there exists an $m \in \N$ and a closed embedding $\rho \colon \G \rightarrow \GL_m$, which is defined over $F_1$, such that there is a non-zero element $w \in F_1^m$ satisfying
$$ \mathcal{H}(F)= \{ g \in \G(F) \ | \ \rho(g)w \in Fw \}. $$ 
\end{thm}

\noindent Note that the rational representation given in \cite{Springer} might not be a closed embedding itself, but it can be turned into one by taking the direct sum with an arbitrary closed embedding defined over $F_1$.

\begin{Def}
Assume that $\Oabs/\mfrak$ embeds into $K$. Then we can extend the canonical homomorphism $\kappa_{|\cdot|} \colon \Oabs \rightarrow \Oabs/\mfrak$ to a ring homomorphism 
\[\kappa_{|\cdot|} \colon \Oabs[[t]] \rightarrow (\Oabs/\mfrak)[[t]] \rightarrow K[[t]], \] by 
setting $\kappa_{|\cdot|}(\sum_{i=0}^{\infty}a_it^i)=\sum_{i=0}^{\infty}\kappa_{|\cdot|}(a_i)t^i$ for any $a_i \in \Oabs$. Note that $\kappa_{|\cdot|}$ commutes with the action of $\phi_q$ on $K[[t]]$.
\end{Def}

In Section \ref{EX} we constructed fundamental matrices $Y \in \GL_n(L)\cap \Mn(K\{t\})$.  
We will eventually need $Y$ to be contained in $\G(K[[t]])$. Of course we still want to stay inside $L$ (to ensure that we have no new constants) so we are looking for fundamental solution matrices contained in $\G(L\cap K[[t]])$. Note that $L\cap K[[t]]=\{\frac{f}{g} \ | \ f, g \in K\{t\}, t\nmid g \}\supsetneq K\{t\}$, for instance $(1-t)^{-1}$ is contained in $L\cap K[[t]]$ but $(1-t)$ is not invertible inside $K\{t\}$. 

\begin{thm}\label{ubthm}
Assume that $\Oabs/\mfrak$ embeds into $K$.
Let $\G \leq \GL_n$ be a connected linear algebraic group defined over $\Fq$. Let further $D=\sum_{l=0}^{\infty} D_lt^l \in \G(\Oabs[[t]])$ be such that $||D_l||<1$ for all $l>0$. Assume that there exists a matrix $Y \in \GL_n(\Oabs[[t]])\cap \Mn(\Oabs\{t\})$ satisfying $D\phi_q(Y)=Y$. Then there exists a $Y' \in \G(L \cap \Oabs[[t]])$ with $D\phi_q(Y')=Y'$. 
\end{thm}
\begin{proof}
For any matrix $A \in \Mn(\Oabs[[t]])$, we set $\tilde A:=\kappa_{|\cdot|}(A)$ and similarly for vectors over $\Ocal[[t]]$ and scalars in $\Ocal[[t]]$. \\
\\
By assumption, we have $\tilde D \in \G(K)$, i.e. no $t$ appears. As $K$ is algebraically closed, the Lang isogeny (see \cite[V.16.4]{Borel}) asserts that there exists an $X \in \G(K)$ satisfying $\tilde D\phi_q(X)=X$. Now $Y$ is contained in $\GL_n(\Oabs[[t]]) \cap \Mn(\Oabs\{t\})$, hence $\tilde Y \in \GL_n(K[[t]])\cap \Mn(K[t]) \subseteq \GL_n(K(t))$. As $\kappa_{|\cdot|}$ and $\phi_q$ commute, we have $\tilde D \phi_q(\tilde Y)=\tilde Y$. Then $C:=\tilde Y^{-1}X$ is contained in $\GL_n(C_{K(t)})=\GL_n(\Fq(t))$.\\
\\
We set $Y':=YC$. Clearly, $D\phi_q(Y')=Y'$ holds since $C$ has constant entries. We claim that $Y'$ is contained in $\G(L\cap \Oabs[[t]])$. First of all, $Y$ has entries in $\Oabs\{t\} \subseteq L$ and $C$ has entries in $\Fq(t)\subseteq L$, hence $YC \in \GL_n(L)$. Also, $\tilde Y \in \GL_n(K[[t]])$ and $X \in \GL_n(K)$, hence $C=\tilde Y ^{-1} X \in \GL_n(K[[t]])$. We conclude $C \in \GL_n(\Fq(t))\cap \GL_n(K[[t]]) \subseteq \GL_n(\Fq[[t]]) \subseteq \GL_n(\Oabs[[t]])$, thus $Y'=YC$ is also contained in $\GL_n(\Oabs[[t]])$. Therefore, it suffices to show that $Y':=YC$ is contained in $\G(\overline{K((t))})$.\\
\\
By the Chevalley Theorem \ref{chevalley}, there exists a closed embedding $\rho \colon \GL_n \rightarrow \GL_m$ defined over $\Fq$ and a non-zero element $w \in \Fq ^{\, m}$ such that 
\begin{equation}\label{glgub1}
\G(\overline{K((t))})=\{g \in \GL_n(\overline{K((t))}) \ | \  \rho(g)w \in \overline{K((t))}\cdot w \}.
\end{equation} By multiplying $w$ by a suitable element in $\Fq^{\, \times}$, we may assume that there exists a $j \leq m$ such that $w_j=1$. \\
Note that $\rho$ commutes with both $\phi_q$ and $\kappa_{|\cdot|}$, as these both act trivially on $\Fq$. Also note that whenever a matrix $A$ is contained in $\GL_n(\Oabs[[t]])$, $\rho(A)$ will be contained in $\GL_m(\Oabs[[t]])$, as $\rho$ is defined over $\Fq \subseteq \Oabs$, hence both $\rho(A)$ and $\rho(A^{-1})$ have entries in $\Oabs[[t]]$. \\
\\
We will show that there exist $v \in \Fq[[t]]^m$ and $\mu \in \Oabs[[t]]^{\times}$ such that 
\begin{equation} \label{glgub2}
\rho(Y'^{-1})w=\mu v 
\end{equation} holds. If this is true, we will have 
\begin{eqnarray*}
\mu^{-1}\rho(Y'^{-1})w&=&v=\kappa_{|\cdot|}(v) =\kappa_{|\cdot|}(\mu^{-1}\rho(Y'^{-1})w) =\tilde \mu^{-1} \rho(\tilde Y'^{-1})\tilde w \\
&=&\tilde \mu^{-1} \rho(\tilde C^{-1} \tilde Y^{-1})w =\tilde \mu^{-1} \rho(C^{-1}\tilde Y^{-1})w =\tilde \mu^{-1} \rho(X^{-1})w,
\end{eqnarray*} where we repeatedly used that $\kappa_{|\cdot|}$ acts trivially on $\Fq[[t]]$. Now $X^{-1} \in \G(K)$, hence $\rho(X^{-1})w \in Kw$ by (\ref{glgub1}). Also, $\tilde \mu \in K[[t]]^{\times}$ (as $\mu \in \Oabs[[t]]^{\times}$) so we conclude
\[\rho(Y'^{-1})w=\mu \tilde \mu^{-1} \rho(X^{-1})w \in K[[t]]w \] which implies that $(Y')^{-1}$ and hence $Y'$ is contained in $\G(\overline{K((t))})$ (see (\ref{glgub1})).\\
\\
It remains to show that there exist $v \in \Fq[[t]]^m$ and $\mu \in \Oabs[[t]]^{\times}$ satisfying Equation (\ref{glgub2}).
First note that as $D \in \G(\Oabs[[t]])\subseteq \G(\overline{K((t))})$, Equation (\ref{glgub1}) implies that there exists a $\lambda \in \overline{K((t))}$ satisfying
\[ \rho(D)w=\lambda w. \] We have $\rho(D) \in \GL_m(\Oabs[[t]])$, hence $\lambda=\lambda w_j=(\rho(D)w)_j \in \Oabs[[t]]$, as  $w_j=1$ and $w \in \Fq^{\, m} \subseteq \Oabs^m$. Similarly, $\lambda^{-1}=\lambda^{-1}w_j=(\rho(D)^{-1}w)_j \in \Oabs[[t]]$, hence $\lambda$ is contained in $\Oabs[[t]]^{\times}$. We set \\ $v':=\rho(Y'^{-1})w \in \Oabs[[t]]^m$ and compute
\begin{eqnarray}
\phi_q(v')&=&\phi_q(\rho(Y'^{-1})w) =\phi_q(\rho(Y'^{-1}))w 
= \rho(\phi_q(Y'^{-1}))w \nonumber \\
&=& \rho(Y'^{-1}D)w =\rho(Y'^{-1})\rho(D)w = \lambda v' .\label{glgub3}
\end{eqnarray}
We can fix a $\mu \in \Oabs[[t]]^{\times}$ satisfying $\phi_q(\mu)\mu^{-1}=\lambda$. We define $v:=\mu^{-1}v'=\mu^{-1}\rho(Y'^{-1})w.$ Then $v \in \Oabs[[t]]^m$, and by Equation (\ref{glgub3}), we have \[\phi_q(v)=\phi_q(\mu^{-1})\phi_q(v')=\phi_q(\mu^{-1})\lambda v'=v, \] hence $v \in \Fq[[t]]^m$ and $(v,\mu)$ satisfy Equation (\ref{glgub2}) by definition. 
\end{proof}

\subsection{Lower Bounds}\label{LB}
We develop a lower bound criterion as follows. Let $(M,\Phi)$ be a difference module over ($\Fq(s,t)$, $\phi_q$) with representing matrix $D \in \GL_n(\Fq(s,t))$. Let $\pfrak$ be a place of $\Fq(s)$ of degree $d\geq 1$ such that $D \in \GL_n(\ofrak[t]_{(t)})$, where $\ofrak$ denotes the valuation ring corresponding to $\pfrak$. Then $\Hcal$ contains a certain conjugate of the reduction of $D\phi_q(D)\cdots\phi_q^{d-1}(D)$ modulo $\pfrak$ (see Theorem \ref{uschr}). The criterion developed here actually works in the more general case $k(t)\supset \Fq(t)$ with $k$ a (not necessarily discretely) valued field with finite residue field. \\
\\
 The idea to work with reductions at some places $\pfrak$ to obtain elements of the Galois group up to conjugacy is inspired by finite Galois theory. Every finite Galois extension of $\Fq(s)$ is the Picard-Vessiot ring of a difference module over $(\Fq(s),\phi_q)$ (note that $\phi_q$ restricts to the ordinary Frobenius endomorphism on $\Fq(s)$).  In \cite{Matzat}, Matzat gave a lower bound criterion for these kind of difference modules (so called finite Frobenius modules) using reductions of the representing matrix $D_0 \in \GL_n(\Fq(s))$ from $\Fq(s)$ to $\Fq$ - this can be seen as a ``linear Dedekind criterion''.   

\subsubsection{Setup for Specialization} \label{setupspec}
In addition to the notation established in section \ref{notation}, we will use the following notation in this section. 

\begin{description}
\item[$d$] a fixed number $d \in \N$.
 \item[$(\ofrak,\pfrak)$] a valuation ring $\ofrak$ inside $k$ with maximal ideal $\pfrak$ such that the residue field $\ofrak/\pfrak$ is isomorphic to $\Fqd$. We do not assume $\ofrak$ to be discrete.
\item[$\Gamma$]  the corresponding ordered abelian group $\Gamma=k^{\times}/\ofrak^{\times}$. 
\item[$(\Ocal,\Pcal)$] an extension of $(\ofrak,\pfrak)$ to $\ksep$.
\item[$\Gamma'$]  the corresponding ordered abelian group $\Gamma':=(\ksep)^{\times}/{\Ocal^{\times}}$. 
\item[$\nu$] the corresponding valuation $\nu \colon \ksep \rightarrow \Gamma' \cup \{ \infty \}$. Note that $\nu$ restricts to $\nu \colon k \rightarrow \Gamma \cup \{ \infty \}$.  
\item[$\kappa$] the residue homomorphism $\kappa \colon \Ocal \rightarrow \Fqbar$. (We have $\Ocal/ \Pcal \cong \Fqbar$, as we assumed $\ofrak / \pfrak \cong \Fqd$.) Note that $\kappa$ restricts to $\kappa \colon \ofrak \rightarrow \Fqd$.
\item[$\nu_t$] the Gauss extension $\nu_t \colon k(t) \rightarrow \Gamma \cup \{\infty \}$  of $\nu$, defined by \\ $\nu_t(\sum_{i=0}^{r}a_it^i)=\min\{\nu(a_i) \ | \ 0 \leq i \leq r \}$ for $a_i \in k$ and $r \in \N$ and extended to fractions of polynomials. 
\item[$(\ofrak_t, \pfrak_t)$] the valuation ring $\ofrak_t$ of $\nu_t$ inside $k(t)$ with maximal ideal $\pfrak_t$. The residue class field equals $\ofrak_t/\pfrak_t\cong \Fqd(t)$ (see \cite[Cor. 2.2.2]{EnglerPrestel}). 
\item[$\Ocal((t))$] the ring of formal Laurent series over $\Ocal$: \\ $\Ocal((t)):=\{\sum_{i=r}^{\infty} a_it^i \ | \ r \in \Z, a_i \in \Ocal \}=\Ocal[[t]][t^{-1}]$. 
\item[$\Ocal_t$] the subring of $\ksep((t))$ generated by $\ofrak_t$ and $\Ocal((t))$. Since both $\Ocal((t))$ and $\ofrak_t$ are $\phi_q$-stable inside $\ksep((t))$, $\Ocal_t$ is $\phi_q$-stable, as well. Also, note that $\Fq(t) \subseteq \Ocal((t)) \subseteq \Ocal_t$. 
\item[$(K((t)),\phi_q)$] we define $\phi_q$ on $K((t))$ (and any subring thereof) by setting $\phi_q(\sum_{i=r}^{\infty}a_it^i)=\sum_{i=r}^{\infty}\phi_q(a_i)t^i=\sum_{i=r}^{\infty}a_i^qt^i$ for $r \in \Z$ and $a_i \in K$. This is compatible with the definition on the subfield $L$ of $K((t))$ made in Section \ref{notation}. 
 \end{description}

\begin{rem}
Note that $\ofrak_t$ is not contained in $\Ocal((t))$. Indeed, $\frac{1}{a+t}$ is contained in $\ofrak_t$ but not in $\Ocal((t))$, if $a \in \pfrak=\ofrak \backslash  \ofrak^{\times} $. 
\end{rem}

\begin{ex}
Note that $\ofrak/\pfrak \cong \Fqd$ includes restrictions on $k$ which had been an arbitrary subfield of $K$, before. For instance $k$ cannot equal $\Fqbar(s)$ anymore, since $\Fqbar$ can be embedded into the residue field of any valuation on $\Fqbar(s)$. In our application, $k=\Fq(s)$ with $\pfrak$ a place of degree $d$. However, the results from this chapter could also be applied in more general situations such as $k=\Fq(s_1,\dots s_r)$ with a rank-$r$-valuation $\nu$ and finite residue class field.
\end{ex}

\subsubsection{Specializing Fundamental Matrices}

\begin{lemma}\label{separablesolutions}
Consider a system 
$AY^q-Y+B=0$ of polynomial equations over $\ksep$ for an $A \in \GL_n(\ksep)$ and $B \in \Mn(\ksep)$, where $Y$ denotes a matrix consisting of $n^2$ indeterminates. Let $Y \in K^{n\times n}$ be a solution. Then all entries of $Y$ are contained in $\ksep$.
\end{lemma}
\begin{proof}
This follows directly from the Jacobian criterion Proposition VIII.5.3. in \cite[Part II]{Lang}.
\end{proof}

\begin{prop}\label{Spezialisierungsprop}
Let $D \in \GL_n(\ofrak[[t]])$ and $Y \in \GL_n(K[[t]])$ be such that $D\phi_q(Y)=Y$. Then $Y$ is contained in $\GL_n(\Ocal[[t]])$.
\end{prop}
\begin{proof} We can write $D=\sum_{l=0}^{\infty}D_lt^{l}$ with 
$D_0 \in \GL_n(\ofrak)$ and $D_l \in \Mn(\ofrak)$ for all $l >0$ and $Y=\sum_{l=0}^{\infty}Y_lt^l$ with 
$Y_0 \in \GL_n(K)$ and all $Y_l \in \Mn(K)$ for $l>0$. Recall that $D \phi_q(Y)=Y$ is equivalent to 
\begin{equation}\label{koeffvglglg2}
D_0Y_{l}^q+D_{1}Y_{l-1}^q+\dots+D_{l}Y_0^q=Y_{l} \text{ for all } l \in \N.
\end{equation} 
Inductively, we see that all entries of $Y_l$ are contained in $\ksep$, by Lemma \ref{separablesolutions}. By induction, $Y_l$ then satisfies an equation of the form $Y_l^q=D_0^{-1}Y_l+A$, where $A$ has entries inside $\Ocal$, hence $\Ocal[(Y_l)_{i,j} \ | \ 1 \leq i,j \leq n]$ is finitely generated as an $\Ocal$-module. Therefore, all entries of $Y_l$ are integral over $\Ocal$ and thus contained in $\Ocal$. It follows that $Y$ is contained in $\Mn(\Ocal[[t]]) \cap \GL_n(\ksep[[t]])$ and $D_0Y_0^q=Y_0$ implies $\det(Y_0) \in \Ocal^{\times}$, hence $Y \in \GL_n(\Ocal[[t]])$. 
\end{proof}

\begin{prop}\label{defkappa}
We can extend the residue class homomorphism $\kappa \colon \Ocal \rightarrow \Fqbar$ to a homomorphism $$\kappa \colon \Ocal_t \rightarrow \Fqbar((t))$$ such that the following holds:
\begin{enumerate}
\item $\kappa$ commutes with $\phi_q$. 
\item $\kappa$ restricted to $\Ocal((t))$ equals the coefficient-wise application of the residue map $\Ocal \rightarrow \Ocal/\Pcal \cong \Fqbar$ to a Laurent series over $\Ocal$. 
\item $\kappa$ restricts to the residue map $\ofrak_t \rightarrow \ofrak_t/ \pfrak_t \cong \Fqd(t)$ on $\ofrak_t$.
\end{enumerate}
\end{prop}

\begin{proof}
As the residue class homomorphism $\kappa \colon \Ocal \rightarrow \Fqbar$ is a homomorphism, it commutes with $\phi_q$ which is the ordinary Frobenius homomorphism on $\Ocal$. We can extend $\kappa$ to a ring homomorphism $\kappa \colon \Ocal((t)) \rightarrow \Fqbar((t))$ by applying $\kappa$ to the coefficients of the Laurent series over $\Ocal$. Since $\phi_q$ acts on $\Ocal$ coefficient-wise as well, we find that $\kappa$ commutes with $\phi_q$ on $\Ocal((t))$. \\
\\
Let $f$ be contained in $\ofrak_t$. Then $f$ can be written as $\frac{g}{h}$ with $g \in \ofrak[t]$ and $h \in \ofrak[t]\cap \ofrak_t^{\times}$ and the residue map $\ofrak_t \rightarrow \Fqd(t)$ maps $f$ on $\frac{\kappa(g)}{\kappa(h)}$. We have 
\begin{eqnarray*}
 \Ocal_t&=&\{\sum_{i=1}^n f_i g_i \ | \ n \in \N, \ f_i \in \Ocal((t)), \ g_i \in \ofrak_t \} \\
&=&\{\sum_{i=1}^n \frac{f_i}{h_i} \ | \ n \in \N, \ f_i \in \Ocal((t)), \ h_i \in \ofrak[t]\cap \ofrak_t^{\times} \}  \\
&=&\{\frac{f}{h} \ | \  \ f \in \Ocal((t)), \ h \in \ofrak[t]\cap \ofrak_t^{\times} \}.
 \end{eqnarray*}
Setting $\kappa(\frac{f}{h})=\frac{\kappa(f)}{\kappa(h)}$ yields a well-defined homomorphism on $\Ocal_t$ with all the desired properties. 
\end{proof}

\subsubsection{A Lower Bound Theorem}
\begin{thm}\label{uschr}
 Let $\mathcal{G}\leq \GL_n$ be a linear algebraic group defined over $\Fq(t)$. Let $(M, \Phi)$ be an $n$-dimensional $\phi_q$-module over $k(t)$ with representing matrix $D \in \G(k(t) \cap \ofrak[[t]])$. Assume that there exists a fundamental matrix $Y \in \G(K[[t]])$ for $M$ generating a separable Picard-Vessiot extension $E/k(t)$ of $M$. Let $\mathcal H \leq \G$ be the Galois group-scheme of $M$ corresponding to the Picard-Vessiot ring $R:=k(t)[Y,Y^{-1}] \subseteq E$. Then $\mathcal{H}(\Fq[[t]])$ contains a $\G(\Fqbar[[t]])$-conjugate of $\kappa(D\phi_q(D)\dots \phi_{q^{d-1}}(D))$. \\ (More precisely, the conjugating matrix can be chosen as $\kappa(Y) \in \G(\Fqbar[[t]])$).
\end{thm}

\begin{proof} We abbreviate $F:=k(t)$ throughout this proof.\\
By Theorem \ref{representability}, $\Hcal$ is a linear algebraic group and it is a subgroup of $\G$ by Proposition \ref{closedemb}.
 We apply Theorem \ref{chevalley} to $\G$ and get a closed embedding $\rho \colon \G \rightarrow \GL_m$ defined over $\Fq(t)$ and a non-zero $w \in \Fq(t)^m$ such that
\begin{equation}\label{chevkonstr}
\mathcal{H}(S)=\{ g \in \G(S) \ | \ \rho(g)w \in S \cdot w \}
\end{equation} holds for all $\Fq(t)$-algebras $S$.
 We now blow up $M$ to an $m$-dimensional version $\tilde M$ in order to be able to apply Lemma \ref{eindim_UM}. Let $\tilde M$ be an $m$-dimensional difference module over $F$ such that its representing matrix with respect to a fixed basis is given by $\rho(D) \in \GL_m(F)$. Then $\tilde{Y}:=\rho(Y)$ is a fundamental solution matrix for $\tilde M$, since $\rho$ is defined over $\Fq(t)$ and thus $\phi_q(\rho(Y))=\rho(\phi_q(Y))$. All entries of $\tilde Y$ are contained in $R$ and furthermore, the entries of $Y$ are contained in $F[\tilde Y, \tilde{Y}^{-1}]$, since $\rho$ is a closed embedding defined over $\Fq(t) \subseteq F$. Hence $R=F[\tilde Y, \tilde Y^{-1}]$ is also a Picard-Vessiot ring for $\tilde M$ and the Galois group scheme is $\rho(\mathcal{H})\leq \GL_m$ in its natural representation as given in Proposition \ref{closedemb}. By construction, $w$ spans a $\rho(\mathcal{H})$-stable line (see Equation (\ref{chevkonstr})) and thus there exists an $\alpha \in E^{\times}$ such that 
\begin{equation}\label{defalphainuschr}
v=\alpha\tilde Yw
\end{equation}
 is contained in $F^m$ and $N=Fv$ defines a difference submodule of $\tilde M$, by Lemma \ref{eindim_UM}. This means that there exists a $\lambda \in F$ such that $\rho(D)\phi_q(v)=\lambda v$. \\
As $k(t)\cap \ofrak[[t]] \subseteq \ofrak_t$, $D$ is contained in $\GL_n(\ofrak_t)$. Since $\rho$ is defined over $\Fq(t)\subseteq \ofrak_t$, $\rho(D)$ and $\rho(D^{-1})$ both have coefficients in $\ofrak_t$ and thus $\rho(D)$ is contained in $\GL_m(\ofrak_t)$. \\
Now fix an $i \leq m$ such that the $i$-th coordinate $v_i$ of $v$ has minimal valuation among all coordinates of $v$ (with respect to $\nu_t$ and the order on $\Gamma$). After dividing $v$ by $v_i$, we may assume $v_i=1$.
 Thus $\lambda=\lambda\cdot v_i=(\rho(D)\phi_q(v))_i$ is contained in $\ofrak_t$ since $\ofrak_t$ is $\phi_q$-stable. Also, $(\lambda)^{-1}=(\lambda)^{-1}\phi_q(v)_i=(\rho(D^{-1})v)_i$ is contained in $\ofrak_t$, so $\lambda$ is in fact contained in $\ofrak_t^{\times}$. Overall, we got $\rho(D) \in \GL_m(\ofrak_t)$, $v \in \ofrak_t^m$ and $\lambda \in \ofrak_t^{\times}$, hence we may specialize them to $\kappa(\rho(D)) \in \GL_m(\Fqd(t))$, $\kappa(v) \in \Fqd(t)^m$ and $\kappa(\lambda) \in \Fqd(t)^{\times}$, by Proposition \ref{defkappa}.\\
\\ 
 We denote from now on the (coordinate-wise) application of $\kappa$ to a matrix $A$ with entries in $\Ocal_t$ by $\overline{A}$ and similarly for vectors with entries in $\Ocal_t$ and scalars in $\Ocal_t$. Applying $\kappa$ to both sides of $\rho(D)\phi_q(v)=\lambda v$ coordinate-wise yields:
\begin{equation}\label{glgmitrho2}  \overline{\rho(D)}\phi_q(\overline{v})=\overline{\lambda}\cdot \overline{v},
\end{equation}
 where we used that $\kappa$ and $\phi_q$ commute (see Proposition \ref{defkappa}) to get $\overline{\phi_q(v)}=\phi_q(\overline{v})$. 
 Note that $\rho$ commutes with the coordinate-wise application of $\kappa$ to an element in $\G(\Ocal_t)$, since $\rho$ is defined over $\Fq(t)$ and $\kappa$ restricts to the identity on $\Fq(t)$. In particular, $\overline{\rho(D)}=\rho(\overline{D})$ holds and we obtain $\rho(\overline{D})\cdot\phi_q(\overline{v})=\overline{\lambda}\cdot \overline{v}.$
Inductively, we obtain 
\begin{equation}\label{glgmitrho3twist}  \rho(\overline{D})\phi_q(\rho(\overline{D}))\cdots \phi_{d-1}(\rho(\overline{D}))\cdot \overline{v}=\overline{\lambda}\phi_q(\overline{\lambda})\cdots\phi_{q^{d-1}}(\overline{\lambda}) \cdot \overline{v},
\end{equation} where we used $\phi_{q^d}(\overline{v})=\overline{v}$.
Proposition \ref{Spezialisierungsprop} implies that $Y$ has entries in $\Ocal[[t]]$, hence $\rho(Y)$ is contained in $\GL_m(\Ocal((t)))\subseteq \GL_m(\Ocal_t)$, as $\rho$ is defined over $\Fq(t)\subseteq \Ocal((t))$. There exists a $j \leq m$ such that $w_j \in \Fq(t)^{\times} \subseteq \Ocal_t^{\times}$. Then Equation (\ref{defalphainuschr}) implies 
$1=v_i=\alpha\cdot (\rho(Y)\cdot w)_i$ and $(\rho(Y)^{-1}\cdot v)_j=\alpha\cdot w_j$ 
and we deduce that $\alpha$ is contained in $\Ocal_t^{\times}$.
It can thus be specialized to an element $\overline{\alpha}=\kappa(\alpha) \in \Fqbar((t))^{\times}$.
 We may apply $\kappa$ to both sides of Equation (\ref{defalphainuschr}) to get
\begin{equation}\label{glgmitmu1}  \overline{v}=\overline{\alpha} \cdot \overline{\rho(Y)} \cdot \overline{w}=\overline{\alpha} \cdot \rho(\overline{Y}) \cdot w.
\end{equation} (Note that at this point we applied $\kappa$ simultaneously to elements in $\Ocal((t))$ and $\ofrak_t$ which is why we had to construct $\kappa$ on the somewhat peculiar ring $\Ocal_t$ in Proposition \ref{defkappa}.) \\
\\
Abbreviate $\hat D=D\phi_q(D)\cdots \phi_{q^{d-1}}(D)$ and $\hat \lambda=\lambda\phi_q(\lambda)\cdots \phi_{q^{d-1}}(\lambda)$.
Then Equation (\ref{glgmitrho3twist}) translates to 
\begin{equation}\label{glgmitrho4twist}
\rho(\overline{\hat D})\overline{v}=\overline{\hat \lambda}\overline{v}.
\end{equation}
We now consider $\overline{Y}^{-1} \cdot \overline{\hat D} \cdot \overline{Y}=\kappa(Y^{-1}\hat DY)$ which is contained in $\G(\Fqbar[[t]])$, since $\G$ is defined over $\Fq(t)$ and $\kappa$ acts trivially on $\Fq(t)$. We use Equation (\ref{glgmitmu1}) and (\ref{glgmitrho4twist}) to compute
\begin{eqnarray*}
 \rho(\overline{Y}^{-1} \cdot \overline{\hat D} \cdot \overline{Y})\cdot w
&=& \rho(\overline{Y^{-1}})\rho(\overline{\hat D})\rho(\overline{Y})w\\
&=&\overline{\alpha}^{-1} \rho(\overline{Y}^{-1})\rho(\overline{\hat D})\overline{v} \\
&=&\overline{\alpha}^{-1} \overline{\hat \lambda} \rho(\overline{Y}^{-1})\overline{v}\\
&=&\overline{\hat \lambda}\cdot w.
\end{eqnarray*} 
It follows that $\overline{Y}^{-1} \cdot \overline{\hat D} \cdot \overline{Y}$ is contained in $\mathcal{H}(\Fqbar[[t]])$ (see (\ref{chevkonstr})). It remains to show that $\overline{Y}^{-1} \cdot \overline{\hat D} \cdot \overline{Y}$ has entries in $\Fq[[t]]$. To see this, recall that $D\phi_q(Y)=Y$ holds, hence $\overline{D}\phi_q(\overline{Y})=\overline{Y}$ and $\phi_q(\overline{Y})^{-1}=\overline{Y}^{-1} \cdot \overline{D} $. We compute
\begin{eqnarray*}
 \phi_q(\overline{Y}^{-1}  \overline{\hat D}  \overline{Y})
&=& \phi_q(\overline{Y}^{-1})  \phi_q(\overline{D}) \cdots \phi_{q^{d}}(\overline{D}) \phi_q(\overline{Y})\\
&=& \phi_q(\overline{Y}^{-1})  \phi_q(\overline{D}) \cdots \phi_{q^{d-1}}(\overline{D})\overline{D} \phi_q(\overline{Y})\\
&=& \overline{Y}^{-1}\overline{D}  \phi_q(\overline{D}) \cdots \phi_{q^{d-1}}(\overline{D})\overline{Y}\\
&=&\overline{Y}^{-1} \cdot \overline{\hat D} \cdot \overline{Y},
\end{eqnarray*} where we used $\overline{D} \in \GL_n(\Fqd(t))$. Hence $\overline{Y}^{-1} \cdot \overline{\hat D} \cdot \overline{Y}$ has entries in $\Fqbar[[t]]^{\phi_q}=\Fq[[t]]$. 
\end{proof}

\begin{ex}
If $k=\Fq(s)$ and $\pfrak=(s-\alpha)$ is a finite place of degree $1$ ($\alpha \in \Fq$), then the Galois group scheme $\Hcal$ contains a conjugate of the specialized matrix $D_\alpha \in \G(\Fq(t))$ obtained by replacing each $s$ occurring in $D \in \G(\Fq(s,t))$ by $\alpha$. 
\end{ex}

\section{Generating Split Reductive Groups}\label{GE}
The lower bound criterion (Theorem \ref{uschr}) provides us with elements that are contained in the Galois group scheme up to conjugacy. In this section, we give generators of connected, reductive and $\Fq$-split groups that allow a certain conjugacy. 
\begin{prop}\label{genprop}
 Let $K_1$ be an infinite field and let $\G \leq \GL_n$ be a connected linear algebraic group defined over $K_1$ such that either $K_1$ is perfect or $\G$ is reductive. Let further $K_2/K_1$ be a field extension and consider the field of formal Laurent series $K_2((t))$ over $K_2$. If $\Hcal \subset \G$ is a closed subvariety defined over $K_2((t))$ such that for all $g \in \G(K_1)$ there exists an $h \in \Hcal(K_2[[t]])$ of the form $h=g+M_1t+M_2t^2+\dots$ for some $M_i \in M_n(K_2)$, then $\Hcal=\G$ holds. 
\end{prop}
\begin{proof} First of all, note that $\G(K_1)$ is dense in $\G$, as we assumed that either $K_1$ is perfect or $\G$ is reductive (see \cite[18.3]{Borel}).\\
Set $m=n^2+1$. Then $\G$ is a closed subvariety of affine $m$-space, since $\G\leq \GL_n$ holds. Let $K_t:=\overline{K_2((t))}$ be an algebraic closure of $K_2((t))$. We consider the vanishing ideals $I(\G)$ and $I(\Hcal)$ of $\G$ and $\Hcal$ inside $K_t[X_1,\dots,X_m]$. Assume that $\Hcal$ is strictly contained in $\G$, i.e., $I(\Hcal) \supsetneq I(\G)$. Now $I(\Hcal)$ is generated by finitely many elements inside $K_2((t)) [{X_1},\dots,{X_m}]$ and we conclude that at least one of them cannot be contained in $I(\G)$. Let $f \in K_2((t)) [{X_1},\dots,{X_m}]$ be such an element, i.e., $f \in I(\Hcal) \backslash I(\G)$. After multiplying by a suitable power of $t$, we may assume that $f$ is contained in $K_2[[t]][{X_1},\dots,{X_m}]\subset K_2[{X_1},\dots,{X_m}][[t]]$. Hence there exist elements $f_{j} \in K_2[{X_1},\dots,{X_m}]$ such that 
\[f=\sum\limits_{j=0}^{\infty} f_{j}t^j. \] As $\G(K_1)$ is dense in $\G$, there exists a $g \in \G(K_1)$ with $f(g) \neq 0$. It follows that there exists a $j \in \N$ such that $f_{j}(g) \neq 0$. Let $j_0 \in \N$ be minimal such that there exists a $g \in \G(K_1)$ with $f_{j_0}(g) \neq 0$. Hence $f_0,\dots,f_{j_0-1}$ vanish on all $\G(K_1)$ and are thus contained in $I(\G)$. Now consider
\[f':=t^{-j_0}(f-\sum_{j=0}^{j_0-1}f_jt^j)=f_{j_0}+f_{j_0+1}t+f_{j_0+2}t^2+\cdots. \]  As $f \in I(\Hcal) \backslash I(\G)$ and 
$\sum_{j=0}^{j_0-1}f_jt^j \in I(\G)$, we have $f' \in I(\Hcal)\backslash I(\G)$, as well. By definition of $j_0$, there exists a $g \in \G(K_1)$ such that $f_{j_0}(g) \neq 0$. By assumptions, there exists an $h=g+M_1t+M_2t^2+\dots \in \Hcal(K_2[[t]])$ for some $M_i \in M_n(K_2)$, i.e., $g$ occurs as the constant term of an element contained in $\Hcal(K_2[[t]])$. We compute
\[0=f'(h)=\sum\limits_{j=0}^{\infty} f_{j+j_0}(h)t^j \in K_2[[t]]  \] and compare the constant terms of both sides. The constant term of the right hand side equals the constant term of $f_{j_0}(h)$ which in turn equals $f_{j_0}(g)$, hence $0=f_{j_0}(g)$, a contradiction. Hence $\Hcal$ cannot be strictly contained in $\G$.
\end{proof}

\begin{thm}\label{arbgen}
Let $\G$ be a connected and reductive linear algebraic group defined over $\Fq$. Assume further that $\G$ splits over $\Fq$, i.e., there exists a maximal torus $T$ of $\G$ that is defined over $\Fq$ and splits over $\Fq$. Let $\Hcal$ be a closed subgroup of $\G$ defined over $\Fqbar((t))$ that contains a conjugate $T^A$ of $T$ for some $A \in \G(\Fq+t\Fqbar[[t]])$ and such that every $g \in \G(\Fq)$ occurs as the constant part of an element inside $\Hcal(\Fqbar[[t]])$. Then $\Hcal=\G$. In particular, $<T^A, \G(\Fq)>$ is dense in $\G$ for any $A \in \G(\Fq+t\Fqbar[[t]])$. 
\end{thm}
\begin{proof}
 By Proposition \ref{genprop} (with $K_1=K_2=\Fqbar$), it is sufficient to show that for any $g \in \G(\Fqbar)$, there exists an element $h \in \Hcal(\Fqbar[[t]])$ with constant part $g$. \\
\\
As the constant part $A_0$ of $A$ is contained in $\G(\Fq)$, the maximal torus $T^{A_0}$ is defined over $\Fq$ and also splits over $\Fq$. Let $\Phi(\G,T^{A_0})$ denote the set of roots with respect to $T^{A_0}$ and for $\alpha \in \Phi(\G,T^{A_0})$, let $U_\alpha$ be the root subgroup corresponding to $\alpha$. Since $T^{A_0}$ splits over $\Fq$, all root subgroups are defined over $\Fq$ and we have isomorphisms 
\[u_\alpha \colon \mathbb{G}_a \rightarrow U_\alpha \] defined over $\Fq$ for all $\alpha \in \Phi(\G,T^{A_0})$ (see \cite[V.18.7]{Borel} for a proof). Now $\G$ is generated by $T^{A_0}$ together with all root subgroups (see \cite[8.1.1]{Springer}) and as all of these are defined over $\Fq\subseteq \Fqbar$, we obtain
\begin{eqnarray*}
 \G(\Fqbar)&=&<T^{A_0}(\Fqbar), U_\alpha(\Fqbar) \ | \  \alpha \in \Phi(\G,T^{A_0})> \\
		&=&<T(\Fqbar)^{A_0}, U_\alpha(\Fqbar) \ | \  \alpha \in \Phi(\G,T^{A_0})>.
\end{eqnarray*}

 Let now $g$ be contained in $\G(\Fqbar)$. Then there exist an $r \in \N$, roots $\alpha_1,\dots,\alpha_r \in  \Phi(\G,T^{A_0})$ (not necessarily pairwise distinct), $s_1,\dots, s_r \in \Fqbar$ as well as $x_1,\dots, x_{r+1} \in T(\Fqbar)$ such that $g$ can be written as 
\[g=x_1^{A_0}u_{\alpha_1}(s_1)\cdots x_r^{A_0}u_{\alpha_r}(s_r)x_{r+1}^{A_0} .\] Any root $\alpha \in \Phi(\G,T^{A_0})$ is a non-trivial character
$\alpha \colon T^{A_0} \rightarrow \mathbb{G}_m$, hence it is surjective. As $u_\alpha(0)=1$ holds for all $\alpha \in \Phi(\G,T^{A_0})$, we may assume that all $s_1,\dots,s_r$ are contained in $\Fqbar^{\times}$, so there exist elements $y_1^{A_0},\dots,y_r^{A_0} \in T^{A_0}(\Fqbar)$ (that is, $y_1,\dots, y_r$ are contained in $T(\Fqbar)$) such that 
\[s_i=\alpha_i(y_i^{A_0}) \] for $1 \leq i \leq r$. The root subgroup isomorphisms $u_\alpha$ are subject to the relation
\[u_\alpha(\alpha(y)s)=u_\alpha(s)^y  \] for all elements $y$ in the maximal torus and field elements $s$. Therefore, we have
$u_{\alpha_i}(s_i)=u_{\alpha_i}(\alpha_i(y_i^{A_0})\cdot 1)=u_{\alpha_i}(1)^{y_i^{A_0}}$ for all $1 \leq i \leq r$ and thus
\[g=x_1^{A_0}({y_1^{A_0}})^{-1}u_{\alpha_1}(1)y_1^{A_0}\cdots x_r^{A_0}(y_r^{A_0})^{-1}u_{\alpha_r}(1)y_r^{A_0}x_{r+1}^{A_0} .\] 
As all isomorphisms $u_{\alpha_i}$ are defined over $\Fq$, we have $u_{\alpha_i}(1) \in \G(\Fq)$ for all $i \leq r$. By assumptions, there exist elements $h_1,\dots,h_r \in \Hcal(\Fqbar[[t]])$ such that the constant part of $h_i$ equals $u_{\alpha_i}(1)$ for $1 \leq i \leq r$. 
Now consider
\[h:=x_1^{A}(y_1^{A})^{-1}h_1y_1^{A}\cdots x_r^{A}(y_r^{A})^{-1}h_ry_r^{A}x_{r+1}^{A} \in \Hcal(\Fqbar[[t]]).\]  Then the constant part of $h$ equals $g$ (recall that $x_1$ and $y_1$ are contained in $\G(\Fqbar)$). Hence $\Hcal=\G$ holds. \\
\\
As a special case, let $\Hcal \subseteq \G$ be the Zariski closure of $<T^A, \G(\Fq)>$. As $A$ is contained in $\G(\Fqbar((t)))$, we deduce that $T^A \cup \G(\Fq)$ is a closed subset of $\G$ defined over $\Fqbar((t))$. Therefore, $\Hcal$ is defined over $\Fqbar((t))$ as well (see \cite[I.2.1(b)]{Borel}). Hence $\Hcal$ conforms to the assumptions made in this Theorem, and $\Hcal=\G$ follows. 
\end{proof} 

 The lower bound criterion Theorem \ref{uschr} provides us with $\G(\Fqbar[[t]])$-conjugates of certain elements that are contained in the Galois group. Therefore, we have to descend from $\G(\Fqbar[[t]])$-conjugacy to $\G(\Fq+t\cdot \Fqbar[[t]])$-conjugacy in order to be able to apply Theorem \ref{arbgen}.

\begin{prop}\label{descendconj}
Let $\G \leq \GL_n$ be a linear algebraic group defined over $\Fq$. Let $g,h$ be contained in $\G(\Fq+t\cdot \Fqbar[[t]])$. Assume that $g$ is contained in a maximal torus $T$ of $\G$ defined over $\Fq$ and that the centralizer of the constant part $g_0 \in T(\Fq)$ of $g$ equals $T$. If $g$ and $h$ are conjugate over $\G(\Fqbar[[t]])$ then they are already conjugate over $\G(\Fq+t\cdot \Fqbar[[t]])$.
\end{prop}
\begin{proof}
Let $A \in \G(\Fqbar[[t]]))$ be such that $g^A=h$. As $\G$ is defined over $\Fq$, the constant part $A_0$ of $A$ is contained in $\G(\Fqbar)$. Similarly, the constant parts $g_0$ of $g$ and $h_0$ of $h$ are contained in $\G(\Fq)$. Then $g^A=h$ implies $g_0^{A_0}=h_0$ and applying $\phi_q$ on both sides yields $g_0^{\phi_q(A_0)}=g_0^{A_0}$. Applying the Lang isogeny (\cite[V.16.4]{Borel})) to the centralizer $T$ of $g_0$, we obtain an element $y \in T(\Fqbar)$ satisfying $\phi_q(y)y^{-1}=\phi_q(A_0)A_0^{-1}$. Hence $y^{-1}A_0$ is contained in $\G(\Fq)$ and $h=g^{A}=g^{y^{-1}A}$ holds. The constant part of $y^{-1}A$ equals $y^{-1}A_0$ which is $\Fq$-rational. Hence $h$ and $g$ are conjugate over $\G(\Fq+t\cdot \Fqbar[[t]])$.
\end{proof}

\section{The Main Result} \label{GEN}
In this section, we prove that every semisimple, simply-connected group defined over $\Fq$ can be realized as a difference Galois group over $(\Fqist,\phi_{q^i})$ for some $i \in \N$ using a result of Nori (Theorem \ref{Nori}) on finite Galois theory. 
\subsection{Finite Frobenius modules}
\noindent A $\phi_q$-difference module over $(\Fq(s),\phi_q)$ is called a \textit{finite Frobenius module} over $(\Fq(s),\phi_q)$. Any finite Frobenius module has a unique Picard-Vessiot ring inside $\Fqsep$. The Picard-Vessiot ring $E$ is then a finite Galois extension of $\Fq(s)$ which we call the Picard-Vessiot extension. The $\Fq$-rational points of the corresponding (finite) Galois group scheme $\G \leq \GL_n$ are isomorphic to $\Gal(E/F)$ via identifying $\gamma \in \Gal(E/F)$ with $Y^{-1}\gamma(Y) \in \G(\Fq)$, where $Y \in \GL_n(E)$ denotes a fixed fundamental solution matrix (see Proposition \ref{closedemb}). We call $G=\G(\Fq)$ the Galois group of $M$. Every finite Galois extension can be obtained in this way using additive polynomials. Details can be found in \cite{Matzat}. 

\begin{thm}[Nori]\label{Nori}
Let $\G$ be a semisimple, simply-connected linear algebraic group defined over $\Fq$. Then there exists a finite Frobenius module over $(\Fq(s),\phi_q)$ with representing matrix contained in $\G(\Fq(s))$, Picard-Vessiot extension $E/\Fq(s)$ linearly disjoint from $\Fqbar$ over $\Fq$, and Galois group $\G(\Fq)$. 
\end{thm}
\begin{proof}
Nori proved that there exists an absolutely irreducible unramified Galois covering of the affine line with Galois group $\G(\Fq)$ (\cite{Nori}). By Theorem 5.2. in \cite{Matzat}, there exists an effective, finite Frobenius module corresponding to the Galois covering provided by Theorem \ref{Nori}, i.e., the representing matrix can be chosen inside $\G(\Fq(s))$.  The Picard-Vessiot extension $E$ is linearly disjoint from $\Fqbar$ over $\Fq$ since the corresponding Galois covering is absolutely irreducible. 
\end{proof}

The following lower bound criterion for finite Frobenius modules due to Matzat can be found in \cite[Thm 4.5]{Matzat}. 
\begin{thm}[Matzat]\label{Matzat}
 Let $M$ be a finite Frobenius module over $(\Fq(s), \phi_q)$ with representing matrix $D \in \GL_n(\Fq(s))$ and Picard-Vessiot extension $E/\Fq(s)$. We fix a fundamental solution matrix $Y \in \GL_n(E)$. Let $\pfrak$ be a place of degree $d$ of $\Fq(s)$ with corresponding valuation ring $\ofrak\subseteq \Fq(s)$. If $D$ is contained in $\GL_n(\ofrak)$ then the following holds:
\begin{itemize}
\item $E/\Fq(s)$ is unramified at $\pfrak$.
\item For any extension $(\Ocal,\Pcal)$ of $(\ofrak,\pfrak)$ to $E$, $Y$ is contained in $\GL_n(\Ocal)$.
\item The Galois group $\Gal(E/\Fq(s))\leq \GL_n(\Fq)$ of $M$ contains the reduction of 
$Y^{-1}D\phi_q(D)\cdots\phi_{q^{d-1}}(D)Y$ modulo $\Pcal$ for any extension $\Pcal$ of $\pfrak$. 
\end{itemize}
\end{thm}

\noindent The following Proposition provides a converse to Theorem \ref{Matzat}:
\begin{prop}\label{converseMatzat}
Let $M$ be a finite Frobenius module over $(\Fq(s), \phi_q)$ with representing matrix $D \in \GL_n(\Fq(s))$, Picard-Vessiot extension $E/\Fq(s)$, and Galois group $G \leq \GL_n(\Fq)$. We fix a fundamental solution matrix $Y \in \GL_n(E)$. Let $g \in G$. Then there exist infinitely many places $(\ofrak,\pfrak)$ of $\Fq(s)$ such that $D$ is contained in $\GL_n(\ofrak)$ and such that there is an extension $(\Ocal, \Pcal)$ from $\Fq(s)$ to $E$ where $g$ equals the reduction of $Y^{-1}D\phi_q(D)\cdots\phi_{q^{\deg(\pfrak)-1}}(D)Y \mod \Pcal$.
\end{prop}
\begin{proof}
We can write $g$ as $g=Y^{-1}\gamma(Y)$ for an element $\gamma \in \Gal(E/\Fq(s))$. The proof of Theorem \ref{Matzat} (see \cite[Thm 4.5]{Matzat}) implies that we are looking for unramified finite places $\pfrak$ of $\Fq(s)$ with extensions $\Pcal$ to $E$ such that $D \in \GL_n(\ofrak)$ and such that $\gamma^{-1}$ is contained in the decomposition group of $\Pcal/\pfrak$ and acts as the Frobenius $\phi_{q^d}$ on $\Ocal/\Pcal$ where $d$ denotes the degree of $\pfrak$. The Chebotarev Density Theorem (see \cite[Thm 6.3.1]{FriedJarden}) implies that there exist infinitely many such places.
\end{proof}

\begin{prop}\label{ascent}
Let $M$ be a finite Frobenius module over $(\Fq(s), \phi_q)$ with representing matrix $D \in \GL_n(\Fq(s))$, Picard-Vessiot extension $E/\Fq(s)$ and Galois group $G$. Assume that $E$ and $\Fqbar$ are linearly disjoint over $\Fq$. Then for any $i\geq 1$, the finite Frobenius module $M_i$ over $(\Fqis,\phi_{q^i})$ given by 
\[D_i:=D\phi_q(D)\dots\phi_{q^{i-1}}(D) \] has Picard-Vessiot extension $E\Fqi$, and Galois group $G$. 
\end{prop}
\begin{proof}
Let $Y \in \GL_n(E)$ be a fundamental solution matrix for $M$. Hence $D\phi_q(Y)=Y$ which inductively implies $D_i \phi_{q^i}(Y)=Y$, so that $Y$ is a fundamental solution matrix for $M_i$ as well. As $E$ is generated over $\Fq(s)$ by the entries of $Y$, we conclude that $E_i:=E\Fqi$ is generated over $\Fqis$ by the entries of $Y$. Hence $E_i$ is a Picard-Vessiot extension of $M_i$ and as $E$ and $\Fqi$ are linearly disjoint over $\Fq$ by assumption,
we have $\Gal(E_i/\Fqis)=\Gal(E/\Fq(s))=G$. 
\end{proof}

\subsection{Lifting Nori's Theorem}
\begin{lemma} \label{defG}
Let $\G$ be a connected, reductive linear algebraic group defined over $\Fq$ of rank $r$. Assume that there exists a maximal torus $T$ that splits over $\Fq$ with $\Fq$-isomorphism  $\gamma \colon \mathbb{G}_m^r \rightarrow T$. Then there exist irreducible polynomials $p_1,\dots,p_r \in \Fq[t]$ such that if we set $g=\gamma(p_1,\dots,p_r) \in T(\Fq(t))$ the following holds:
\begin{itemize}
 \item $g$ is contained in $T(\Fq[t]_{(t)})$ and $g \equiv I \mod t$.
\item For any $g_0 \in T(\Fq)$, $g_0g$ generates a dense subgroup of $T$. In particular, the centralizer of $g_0g$ inside $\G$ equals $T$.
\end{itemize}
\begin{proof}
Choose pairwise distinct irreducible polynomials $p_1$,$\dots$,$p_r \in \Fq[t]$ with constant terms $1$ and set 
$g:=\gamma(p_1,\dots,p_r).$ Then for any $g_0 \in T(\Fq)$ and every non-trivial character $\chi$ of $T$, we have $\chi(g_0g)\neq1$ and $g_0g$ thus generates a dense subgroup of $T$. 
\end{proof}
\end{lemma}

\begin{thm}\label{result}
Let $\G \leq \GL_n$ be a semisimple, simply-connected linear algebraic group defined over $\Fq$. Then for a suitable $i \in \N$  there exists an $n$-dimensional difference module $M$ over $(\Fqist,\phi_{q^i})$ with a separable Picard-Vessiot ring $R/\Fqist$ and corresponding Galois group scheme isomorphic to $\G$ (as linear algebraic group over $\Fqi\hspace{-0.02cm}(t)$). 
\end{thm}
\begin{proof}
By replacing $q$ by a power $q^i$ we may assume that there exists a maximal torus $T$ of $\G$ that splits over $\Fq$ and such that $T(\Fq)$ contains a regular element $g_0$. Then the dimension of the centralizer $\mathcal{C}_\G(g_0)$ equals $r$, the rank of $\G$. As $\G$ is semisimple and simply-connected, all centralizers of semisimple elements are connected (see \cite[Thm 3.5.6]{Carter}), hence 
\begin{equation} \label{centrg0}
 \mathcal{C}_\G(g_0)=T.
\end{equation}
 Let $D_0' \in \G(\Fq(s))$ be the representing matrix of the finite Frobenius module with Galois group $\G(\Fq)$ coming from Nori's theorem \ref{Nori}. Thanks to the Lang-isogeny, we can fix a fundamental solution matrix $Y_0 \in \G(\Fqsep)$. We apply Proposition \ref{converseMatzat} to $g=g_0$ and obtain a finite place $\pfrak'$ of $\Fq(s)$ and an extension $\Pcal'$ to the Galois extension $E$ corresponding to our finite Frobenius module. We define $i$ to be the degree of $\pfrak'$. We then fix an extension $\Pcal$ of $\Pcal'$ from $E$ to $E \Fqi$ and set $\pfrak=\Pcal\cap \Fqi(s)$. We further define \[D_0=D_0'\phi_q(D_0')\dots \phi_{q^{i-1}}(D_0') \in \G(\Fq(s)).\] Then by Proposition \ref{ascent}, the finite Frobenius module $M_0$ over $(\Fqis, \phi_{q^i})$ given by $D_0$ has Galois group $\G(\Fq)$ and $Y_0$ is still a fundamental solution matrix. By construction, the reduction of $Y_0^{-1}D_0Y_0$ modulo $\Pcal'$ equals $g_0$ and so does the reduction modulo $\Pcal$. Denote the reduction of $D_0$ modulo $\pfrak$ by $\overline{D}_0 \in \G(\Fqi)$. Then $\overline{D}_0$ is conjugate to $g_0$ over $\G(\Fqbar)$ and in fact over $\G(\Fqi)$ by Equation (\ref{centrg0}) (compare Proposition \ref{descendconj}). So we obtain an element $x \in \G(\Fqi)$ with 
\begin{equation}\label{defx}
\overline{D}_0=g_0^x.
\end{equation} 
Fix irreducible elements $p_1,\dots,p_r \in \Fq[t]$ as in Lemma \ref{defG} and set \\ $g=\gamma(p_1,\dots,p_r) \in T(\Fq[t]_{(t)})$ (with $\gamma \colon \mathbb{G}_m^r \tilde \rightarrow T$ defined over $\Fq$). Then $g_0g$ generates a dense subgroup of $T$ and $g\equiv I \mod t$. Fix a finite place $\mathfrak{q}\neq \pfrak$ of $\Fqis$ such that $D_0$ is contained in $\GL_n(\ofrak_\mathfrak{q})$, where $\ofrak_\mathfrak{q}$ denotes the corresponding valuation ring inside $\Fqis$. Let $f_\qfrak \in \Fqi[s]$ be a generator of $\mathfrak{q}$. Recall that $\pfrak$ is of degree $1$ in $\Fqis$, hence there exists an $\alpha \in \Fqi$ such that $\pfrak=(s-\alpha)$. Then $f_\qfrak(\alpha) \in \Fqi^{\times}$, as we assumed $\mathfrak{q}\neq \pfrak$. Let $p_{jl} \in \Fq$ denote the coefficients of $p_j$, i.e.,
\begin{equation*}
p_j=1+\sum_{l=1}^{n_j} p_{jl}t^l \in \Fq[t],
\end{equation*} for all $1 \leq j \leq r$. We set
\begin{equation*}
\tilde p_j=1+\sum_{l=1}^{n_j} p_{jl}\left(\frac{f_\qfrak}{f_\qfrak(\alpha)} \right)^l t^l \in \Fqi[s][t],
\end{equation*} for all $1 \leq j \leq r$. Note that $\tilde p_1,\dots,\tilde p_r$ are invertible inside $\Fqis[t]_{(t)}$, hence we can define
\begin{equation*}
\tilde g:=\gamma(\tilde p_1,\dots,\tilde p_r) \in T(\Fqis[t]_{(t)}).
\end{equation*} We can now define the representing matrix $D \in \G(\Fqist)$ of the desired difference module as
\begin{center}
\framebox{$D=D_0\tilde g^x \in \G(\Fqis[t]_{(t)})$} \end{center} with $D \equiv D_0 \mod t$, as $\tilde g \equiv I \mod t$.
Let $M$ be the corresponding difference module over $(\Fqist,\phi_{q^i})$. \\
\\
We first show that there exists a Picard-Vessiot extension for $M$. Let $|\cdot|$ be the absolute value on $k:=\Fqis$ corresponding to $\mathfrak{q}$ with $|f_\qfrak|=\frac{1}{2}$. We use the corresponding notation (such as $K$, $\Ocal_{|\cdot|}$, $\mfrak$ and $L$) set up in Section \ref{notation}. By construction, the absolute value of the $l$-th coefficient of $\tilde p_i$ is at most $\left(\frac{1}{2}\right)^l$ and the same holds for $\tilde p_i^{-1}$ (written as power series in $t$). Every entry of $\tilde g=\gamma(\tilde p_1,\dots,\tilde p_r)$ is given $\Fq$-polynomially in $\tilde p_1,\dots,\tilde p_r$ and their inverses, hence every entry of the $l$-th coefficient matrix $\tilde g_l$ of $\tilde g$ is bounded by $\left(\frac{1}{2}\right)^l$, as well. We conclude 
\[ || \tilde g_l|| \leq \left(\frac{1}{2}\right)^l   \] for every $l \in \N$. As $x$ is contained in $\G(\Fqi)$, conjugating $\tilde g$ with $x$ is given $\Fqi$-linearly in the entries of $\tilde g$ and thus doesn't affect this convergence. Finally, we assumed $D_0 \in \GL_n(\ofrak_\mathfrak{q})$, hence $||D_0|| =1$ and we conclude
\[ ||D_l||=||D_0\tilde g_l^x|| \leq \left(\frac{1}{2}\right)^l \] for all $l \in \N$ (where $D_l \in \Mn(\Fqis)$ denotes the $l$-th coefficient matrix of $D$). We can now apply Theorem \ref{thmexistenceofsol} (with $\delta=\frac{1}{2}$) and obtain a fundamental solution matrix $Y \in \GL_n(\Ocal_{|\cdot|}[[t]])\cap \Mn(\Ocal_{|\cdot|}\{t\})$. We apply Theorem \ref{ubthm} (note that $\Ocal_{|\cdot|}/\mfrak\cong \Fqbar$ embeds into $K$) and obtain another fundamental solution matrix $Y'$ that is contained in $\G(L\cap \Ocal_{|\cdot|}[[t]])$. Then the constant part $Y_0'$ of this new fundamental solution matrix is contained in $\G(K)$ and it is a fundamental solution matrix for $D_0$. After multiplying $Y'$ from the right with $ Y_0'^{-1}Y_0 \in \G(\Fq)$, we may thus assume that the constant part of $Y'$ equals our previously chosen $Y_0$. From now on, we simply denote $Y'$ by $Y$. Then $R:=\Fqist[Y,Y^{-1}]\subseteq L$ is a Picard-Vessiot ring for $M$ by Theorem \ref{Pap4.2.5}. All entries of $Y$ are actually contained in $\Fqsept$ (and not just in $K((t))$ - compare the proof of Proposition \ref{Spezialisierungsprop}). Now $\Fqsept$ is separable over $\Fqi(s)(t)$, hence $R/\Fqist$ is separable. We conclude that the Galois group scheme $\Hcal$ of $M$ is a linear algebraic group (see Theorem \ref{representability}) defined over $\Fqi\hspace{-0.02cm}(t)$ and it is a closed subgroup of $\G$ by Proposition \ref{closedemb}. \\
\\
We will now use the lower bound criterion \ref{uschr} to show that $\Hcal=\G$. By Theorem \ref{arbgen}, it suffices to show that every element inside $\G(\Fq)$ occurs as a constant term inside $\Hcal(\Fqbar[[t]])$ and that $\Hcal$ contains a $\G(\Fq+t\Fqbar[[t]])$-conjugate of the $\Fq$-split torus $T$. The key point is to show that this is really a $\G(\Fq+t\Fqbar[[t]])$-conjugate and not just a $\G(\Fqi+t\Fqbar[[t]])$-conjugate.\\
\\
First of all, note that for any finite place $\qfrak'$ of $\Fqis$ with valuation ring $\ofrak' \subseteq \Fqi(s)$, the polynomials $\tilde p_1,\dots,\tilde p_r$ are contained in $(\ofrak'[t]_{(t)})^{\times}$. Hence \\ $\tilde g= \gamma(\tilde p_1,\dots,\tilde p_r)$ and also $\tilde g^x$ are contained in $\G(\ofrak'[[t]])$. We conclude that $D$ is contained in $\G(\ofrak'[[t]])$ if and only if $D_0$ is contained in $\G(\ofrak')$. \\
\\
Consider $\qfrak'=\pfrak$ with corresponding valuation ring $\ofrak$. Then $D_0$ is contained in $\G(\ofrak)$ by the choice of $\pfrak$. Let $\Ocal$ be the (non-discrete) valuation ring inside $\Fqsep$ corresponding to a fixed extension $\tilde\Pcal$ of $\pfrak$ and let $\kappa \colon \Ocal[[t]] \rightarrow \Fqbar[[t]]$ denote the coefficient-wise reduction modulo $\tilde \Pcal$. By Theorem \ref{uschr} (with $k=\Fqis$), $\Hcal(\Fqi[[t]])$ contains $h:=\kappa(Y^{-1}DY)$ (since $d=1$, as $\ofrak/\pfrak \cong \Fqi$). We use $\kappa(s)=\alpha$, hence $\kappa(\tilde p_j)=p_j$ for all $j$ to compute
\begin{eqnarray*}
\kappa(D)&=&\kappa(D_0)\kappa(\tilde g)^x \\
&=&\overline{D_0}\gamma(\kappa(\tilde p_1),\dots,\kappa(\tilde p_r))^x \\
&=& g_0^x\gamma(p_1,\dots,p_r)^x \\
&=&(g_0g)^x,
\end{eqnarray*} where we also used Equation (\ref{defx}).
Therefore, $h$ is conjugate to $g_0g$ via $x\cdot\kappa(Y)^{-1} \in \G(\Fqbar[[t]])$. On the other hand, the constant term of $h$ equals the reduction of $Y_0^{-1}D_0Y_0$ at $\Pcal$, which equals $g_0$ by construction. Hence $h$ is contained in $\G(\Fq+t\Fqi[[t]])$ and is thus conjugate to $g_0g \in \G(\Fq[[t]])$ not only over $\G(\Fqbar[[t]])$ but also over $\G(\Fq+t\Fqbar[[t]])$, by Proposition \ref{descendconj}. Let $A \in \G(\Fq+t\Fqbar[[t]])$ be such that $(g_0g)^A=h$. Recall that $g_0g$ generates a dense subgroup of $T$. Hence $h=(g_0g)^A$ generates a dense subgroup of $T^A$, and $\Hcal$ thus contains $T^A$. \\
\\
For the finite part, let $\xi \in \G(\Fq)$ be arbitrary and fix one of the finite places $\pfrak_\xi$ of $\Fqis$ with extension $\Pcal_\xi$ provided by Proposition \ref{converseMatzat} applied to the finite Frobenius module $M_0$ over $(\Fqis, \phi_{q^i})$. Let $\ofrak_\xi$ denote the corresponding valuation ring inside $\Fqi(s)$ and $d_\xi$ the degree of $\pfrak_\xi$.  Then $D_0 \in \G(\ofrak_\xi)$ and thus $D \in \G(\ofrak_\xi[[t]])$. Let further $\tilde\Pcal_\xi$ be an extension of $\Pcal_\xi$ to $\Fqsep$. Then by Theorem \ref{uschr}, $\Hcal(\Fqi[[t]])$ contains 
\[\kappa_\xi(Y^{-1}D\phi_q(D)\dots\phi_{q^{d_\xi-1}}(D)Y), \] where $\kappa_\xi$ denotes the coefficient-wise reduction modulo $\tilde \Pcal_\xi$. Looking at constant parts, we deduce that the reduction of \[Y_0^{-1}D_0\phi_q(D_0)\dots\phi_{q^{d_\xi-1}}(D_0)Y_0 \mod \Pcal_\xi\] occurs as a constant term in $\Hcal(\Fqi[[t]])$. By construction, this reduction equals $\xi$. Hence every element in $\G(\Fq)$ occurs as a constant term inside $\Hcal(\Fqi[[t]])$ which concludes the proof.
\end{proof}

\subsection{Example}
Let now $\G=\SL_n$, assume $q> n(n+1)/2$, and let $T$ be the ($\Fq$-split) diagonal torus inside $\SL_n$. If $\zeta \in \Fq$ is a $(q-1)$-th primitive root of unity, then $T(\Fq)$ contains the regular element 
\[g_0:=\diag(\zeta,\zeta^2,\dots,\zeta^{n-1},\zeta^{-\frac{n(n-1)}{2}}). \] It was shown in $\cite{AM}$ that there exists $f_i \in \Fq[s]$ of the form $f_i=s\alpha_i+(1-s)\beta_i$ for some $\alpha_i,\beta_i \in \Fq$ such that the finite Frobenius module over $(\Fq(s),\phi_q)$ given by 
\[
D_0=\begin{pmatrix}
f_1 & \dots & f_{n-1} & (-1)^{n-1} \\
1&&&\\
&\ddots&&\\
&&1&0 
\end{pmatrix}
 \] has Galois group $\SL_n(\Fq)$. Let $\gamma_1,\dots,\gamma_{n-1}$ be the coefficients of the characteristic polynomial of $g_0$. Fix an element $\alpha \in \Fq\backslash\{0,1 \}$. Then it is easy to see that if we alter $f_i$ to 
\[f_i=s\alpha_i+(1-s)\beta_i+\frac{s(s-1)}{\alpha(\alpha-1)}(\gamma_i-\alpha\alpha_i-(1-\alpha)\beta_i), \] the corresponding Frobenius module $M_0$ over $(\Fq(s), \phi_q)$ has the same Galois group. For this new Frobenius module, there exists a place $\pfrak$ of degree $1$ of $\Fq(s)$, namely $\pfrak=(s-\alpha)$, such that the specialization of $D_0$ at $\pfrak$ is conjugate to $g_0$ over $\G(\Fq)$. Hence the number $i$ in Theorem \ref{result} can be chosen as $i=1$. The elements $p_j$ in Lemma \ref{defG} can be chosen as $p_j=(1+\zeta^jt)$ for $1\leq j \leq n-1$. Following the proof of Theorem \ref{result}, we obtain that the difference module $M$ over $(\Fq(s,t),\phi_q)$ given by 
\begin{center}
\framebox{$D=D_0\cdot \diag(\tilde p_1,\dots,\tilde p_{n-1},(\tilde p_1\cdots \tilde p_{n-1})^{-1})^x $} \end{center} 
has Galois group $\SL_n$ where the elements $\tilde p_j \in \Fq[s,t]$ and $x \in \G(\Fq)$ can also be chosen explicitly: We fix the finite place $\qfrak=(s)$, hence $f_\qfrak=s$ and we can define $\tilde p_j$ as 
\[ \tilde p_j:=1+\zeta^j\frac{s}{\alpha} t \] for $1 \leq j \leq n-1$. Finally, $x \in \SL_n(\Fq)$ is a matrix such that the reduction of $\overline{D_0}$ of $D_0$ at $\pfrak=(s-\alpha)$ equals $g_0^x$. We have 
\[\overline{D}_0 =
\begin{pmatrix}
\gamma_1 & \dots & \gamma_{n-1} & (-1)^{n-1} \\
1&&&\\
&\ddots&&\\
&&1&0 
\end{pmatrix}
 \]   and it is easy to see that $x$ can be chosen as \[x=
\begin{pmatrix} \det(A)^{-1} &&& \\ &1&& \\ &&\ddots& \\ &&&1 
\end{pmatrix}\cdot A\] with $A$ the Vandermonde-matrix corresponding to $(\zeta^{-1},\zeta^{-2},\dots,\zeta^{-n+1},\zeta^{\frac{n(n-1)}{2}})$.

\section{Pre-$t$-Motives}\label{pretm}
In this section, we lift our result from $k(t)$ to $\kt$ to get pre-$t$-motives with semisimple simply-connected Galois groups. Pre-$t$-motives are defined in Definition \ref{pretmotive}. For more information on the theory of $t$-motives, we refer the reader to \cite{Papanikolas} and \cite{Taelman} as well as to the survey articles \cite{BPap} and \cite{Chang}.\\
We first specify our notation as follows.
\begin{description}
\item[$k$:] $k=\Fq(\theta)$, a rational function field.
\item[$|\cdot|_{\infty}$:] the $\infty$-adic valuation on $k$ with $|\theta|_{\infty}=q$.
\item[$K$, $L$:] are as defined in Section \ref{notation} (with respect to $|\cdot|_\infty$).
\item[$\sigma$:] on $\overline{k}$ and $K$, $\sigma$ is the inverse of the Frobenius and $\sigma$ extends to $\overline{k}(t)$, $K\{t\}$ and $L$  by acting coefficient-wise, i.e., $\sigma(t)=t$. Note that $L^{\sigma}=k^{\sigma}=\Fq(t)$ holds.
 \end{description}

\begin{Def} 
A \emph{pre-$t$-motive} is a left $\overline{k}(t)[\sigma,\sigma^{-1}]$-module that is finite dimensional over $\overline{k}(t)$. In other words, a pre-$t$-motive is a difference module $(P,\sigmab)$ over $(\overline{k}(t),\sigma)$ as defined in Definition \ref{defdiffmod}. The notion ``pre-$t$-motive'' depends on $q$, since $\sigma=\phi_q^{-1}$. When considering pre-$t$-motives with respect to different $q$'s at the same time, we will clarify this by calling a pre-$t$-motive corresponding to $\sigma=\phi_q^{-1}$ a \emph{pre-$q$-$t$-motive}. If $q$ has been fixed, a pre-$q^i$-$t$-motive is sometimes called a \emph{pre-$t$-motive of level $i$}. \end{Def}

\begin{Def} \label{pretmotive}
A pre-$t$-motive $(P,\sigmab)$ is called \emph{rigid analytically trivial}, if $P\otimes_{\kbar(t)} L$ has a $\sigmab$-invariant $L$-basis. In other words, $P$ is rigid analytically trivial if and only if there exists a Picard-Vessiot ring of $P$ contained in $L$. 
\end{Def}

The category of rigid analytically trivial pre-$t$-motives is a neutral Tannakian category over $\Fq(t)$ with fiber functor mapping a pre-$t$-motive $(P, \sigmab)$ to the vector space of solutions inside $P\otimes_F L$, i.e. the elements $\alpha \in P\otimes_F L$ with $\sigmab\otimes \sigma(\alpha)=\alpha$ (\cite[3.3.15]{Papanikolas} ). 

\begin{thm}
Let $\G \leq \GL_n$ be a semisimple and simply-connected linear algebraic group defined over $\Fq$. Then there exists an $i \in \N$ and a pre-$q^i$-$t$-motive that is rigid analytically trivial and has Galois group isomorphic to $\G$ as linear algebraic group over $\Fqi\hspace{-0.02cm}(t)$. 
\end{thm}
\begin{proof} We proved in Theorem \ref{result} that there exists an $i \in \N$ and a difference module $M$ over $(\Fqi(s,t), \phi_{q^i}$) with Galois group $\G$. Recall that we constructed the Picard-Vessiot ring inside $L_\qfrak$, where $\qfrak$ denotes a finite place inside $\Fq(s)$ and $L_\qfrak$ denotes the fraction field of $K_\qfrak\{t\}$ with $K_\qfrak$ the completion of an algebraic closure of the completion of $k$ with respect to an absolute value coming from $\qfrak$. We may assume without loss of generality that $\qfrak$ has degree $1$ (by choosing a larger $i$ in the proof of $\ref{result}$), hence $\qfrak$ is of the form $(s-\alpha)$. We can now rename $\theta=\frac{1}{s-\alpha}$, i.e. we replace every occurrence of $s$ in the representing matrix $D \in \GL_n(\Fqis(t))$ of $M$ with $(\theta^{-1}+\alpha)$ and obtain a matrix $\Phi \in \GL_n(\Fqi(\theta)(t))$. Hence we have found a difference module $M$ over $(\Fqi(\theta)(t), \phi_{q^i})$ with Galois group $\G$, fundamental solution matrix $Y \in \GL_n(L)$ and Picard-Vessiot ring $R:=k(t)[Y,Y^{-1}] \subseteq L$. Hence $R$ and $\kt$ are both contained in $L$ and the Galois group is connected, so we can apply Theorem \ref{basechangethm} to conclude that $M\otimes_k \overline{k}$ has Picard-Vessiot ring $R \otimes_k \overline{k}=\kt[Y,Y^{-1}]\subseteq L$ over $\kt$ and also Galois group $\G$.\\
Let $P$ be the pre-$q^i$-$t$-motive given by $\Phi \in \GL_n(\overline{k}(t))$ and set $\Psi=\phi_q(Y) \in \GL_n(L)$. As $Y$ is a fundamental solution matrix for $M$, we have \[\Phi \phi_q(Y)=Y \] which translates to \[\Phi \Psi=\sigma(\Psi). \] Hence $\Psi$ is a rigid analytic trivialization of $P$ and \[R\otimes_k \overline{k}=\kt[Y,Y^{-1}]=\kt[\Psi,\Psi^{-1}]\] is also a Picard-Vessiot ring for $P$. Hence the Galois group schemes of $P$ and $M\otimes_k \overline k$ coincide (they both equal $\underline{\Aut}(R\otimes_k \kbar/\kt)$).  
\end{proof}

 \bibliographystyle{alpha}	
  \bibliography{references}

\noindent \emph{Lehrstuhl f\"ur Mathematik (Algebra), RWTH Aachen University,\\
   52062 Aachen, Germany. email: annette.maier@matha.rwth-aachen.de}

\end{document}